\documentclass{amsart}
\usepackage{amsmath,amsthm,amssymb,amsfonts}
\usepackage{graphicx}
\usepackage{tikz}
  \usepackage{epsfig}
\usepackage[colorlinks=true]{hyperref}

\hypersetup{urlcolor=blue, citecolor=blue}

\def\doi#1{   {\href{http://dx.doi.org/#1}
   {{\mdseries\ttfamily DOI}}}}

\usepackage{graphics}

\def\e{\varepsilon}

\def\<{\langle}             \def\>{\rangle}

\newcommand{\beeq}{\begin{equation}}\newcommand{\eneq}{\end{equation}}
    \newcommand{\be}{\beta}
\newcommand{\de}{\delta}    
\newcommand{\ep}{\epsilon}\newcommand{\vep}{\varepsilon}
  
    \newcommand{\la}{\lambda}
    
\newcommand{\om}{\omega}    
\newcommand{\ga}{\gamma}    
\newcommand{\R}{\mathbb{R}}

\newcommand{\CO}{\mathcal{O}}
\newcommand{\Sp}{\mathbb{S}}

\newenvironment{prf}{\noindent {\bf Proof.} }{\endprf\par}
\def \endprf{\hfill  {\vrule height6pt width6pt depth0pt}\medskip}
\newcommand{\pt}{\partial_t}\newcommand{\pa}{\partial}
\newcommand{\f}{\frac}

\newcommand{\les}{{\lesssim}}\newcommand{\ges}{{\gtrsim}}%{{\succsim}}
\newcommand{\supp}{\,\mathop{\!\mathrm{supp}}}

\def\({\left(}                 \def\){\right)}
\numberwithin{equation}{section}
\newtheorem{thm}{Theorem}[section]
 \newtheorem{cor}[thm]{Corollary}
 \newtheorem{lem}[thm]{Lemma}
 
 \newtheorem{defn}[thm]{Definition}
 \newtheorem{rem}[thm]{Remark}

\title[Lifespan for semilinear wave equations with damping and potential]
{Lifespan estimates for semilinear wave equations with space dependent damping and potential}
\author{Ning-An Lai}
\address{Institute of Nonlinear Analysis and Department of Mathematics\\
		Lishui University\\Lishui 323000, P. R. China}
		\email{ninganlai@lsu.edu.cn}

\author{Mengyun Liu$^{*}$}\thanks{* Corresponding author}
\address{Department of Mathematics\\Zhejiang Sci-Tech University\\Hangzhou 310018, P. R. China }
\email{mengyunliu@zstu.edu.cn}

\author{Ziheng Tu}
\address{School of Data Science, Zhejiang University of Finance and Economics, 310018 Hangzhou, P. R. China }\email{tuziheng@zufe.edu.cn}

\author{Chengbo Wang}
\address{School of Mathematical Sciences\\ Zhejiang University\\Hangzhou 310027, P. R. China}\email{wangcbo@zju.edu.cn }



\date{\today}

\begin{document}
% \begin{CJK*}{UTF8}{gbsn}
\maketitle

\begin{abstract}
In this work, we investigate the influence of
general damping and potential terms on the
blow-up and lifespan estimates for
energy solutions to power-type semilinear wave equations. The
space-dependent damping and potential functions are assumed to
be critical or short range, spherically symmetric perturbation.
The blow up results and the upper bound of lifespan estimates are obtained by the so-called test function method. The key ingredient is to construct
special positive solutions to the linear dual problem
with the desired asymptotic behavior, which is reduced, in turn, to constructing solutions to certain elliptic ``eigenvalue" problems.
\end{abstract}

%\tableofcontents

\section{Introduction}

The purpose of this paper is to investigate the influence of
general damping and potential terms on the
blow-up and lifespan estimates for
energy solutions to power-type semilinear wave equations. The
space-dependent damping and potential functions are assumed to
be critical or short range, spherically symmetric perturbation.

More precisely,
let $n\ge 2$, $p>1$, $D, V\in C(\R^n\backslash\{0\})$,
 we consider the following Cauchy problem of semilinear wave equations, with small data
\begin{equation}
\label{main}
\left\{
\begin{aligned}
& u_{tt} - \Delta u + D(x)u_t+V(x)u = |u|^p, (t, x) \in (0, T)\times\R^n, \\
& u(x,0)=\e f(x), \quad u_t(x,0)=\e g(x)
\end{aligned}
\right.
\end{equation}
Here
$f, g\in C^\infty_c(\R^n)$,  and
the small parameter $\vep>0$ measures the size of the data. As usual, to show blow up, we assume  both $f$ and $g$ are nontrivial, nonnegative and supported in $B_R:= \{x\in\R^n: r\le R\}$ for some $R>0$, where $|x|=r$. In view of scaling, we see that
$D$ and $V$ are critical or short range, if $D=\CO (|x|^{-1})$, $V=\CO(|x|^{-2})$, near spatial infinity.

There have been many evidences that the critical power, $p_c$, for $p$ so that the problem admits global solutions, seems to be related with two kinds of the dimensional shift due to the critical damping and potential coefficients near the spatial infinity. Here we call $p_c$ to be a critical power, if there exists $\de>0$ such that there are certain class of data $(f,g)$ so that
we have blow up for any $\vep>0$ and $p\in (p_c-\de, p_c)$, while there are
for any $p\in (p_c, p_c+\de)$, we have small data global existence for $\vep\in (0, \vep_0)$.

Heuristically, in the sample case $D=d_\infty /r$ and $V=0$, the perturbed linear wave operator for radial solutions is of the following form
$$u_{tt} - \Delta u + D(x)u_t=(\pt^2-\pa_r^2-\frac{n+d_\infty-1}{r}\pa_r)u+\frac{d_\infty}{r}(\pt+\pa_r)u\ .$$
In view of the dispersive nature of the solutions for wave equations, $(\pt+\pa_r)u$ tend to be negligible (good derivative) and thus it
behaves like $n+d_\infty$ dimensional wave equations, which suggests the role of $n+d_\infty$. On the other hand, when we consider the elliptic operator $-\Delta +V$ with $V=v_\infty (1+r)^{-2}$, the asymptotic behavior of the radial solutions seems to be determined by the operator
$\pa_r^2+\frac{n-1}{r}\pa_r-v_\infty r^{-2}$, which is a linear ODE operator of the Euler type and has eigenvalues
\beeq
\rho(v_\infty):=\sqrt{\big(\frac{n-2}{2}\big)^2+v_\infty}-\frac{n-2}{2}, -(n-2)-\rho(v_\infty).
\eneq This suggests the role of $\rho(v_\infty)$.
In conclusion, the heuristic analysis strongly suggests that, under some reasonable assumptions on $D$ and $V$,  we have a critical power given by \beeq p_c=\max(p_S(n+d_\infty), p_G(n+\rho(v_\infty)))\ ,\eneq
where, for $m\in\R$,
\beeq
\ p_G(m)=
\left\{\begin{array}{ll}
%\frac{n+1+\sqrt{(n+1)^2+8(n-1)}}{2(n-1)}
1+\frac{2}{m-1}
& m>1,\\
\infty & m\le 1 ,
\end{array}\right.
d_\infty=\lim_{r\to \infty}rD(r),\
v_\infty=\lim_{r\to \infty}r^2 V(r),
\ \eneq
and
 $p_S(m)$ is related to the Strauss exponent \cite{Strauss}, which is defined to be
\beeq p_S(m)=\left\{\begin{array}{ll}
%\frac{n+1+\sqrt{(n+1)^2+8(n-1)}}{2(n-1)}
\frac{m+1+\sqrt{m^2+10 m-7}}{2(m-1)}
& m>1\\
\infty & m\le 1\ .
\end{array}\right.\eneq
Here, $p_G$ is related to the
Glassey exponent $p_G(n)$ for \[
u_{tt}-\Delta u=|u_t|^p\ , x\in\R^n\ ,
\] or the Fujita exponent $p_F(n)=p_G(n+1)$ for heat or damped wave equations
\[
u_{tt}+u_t-\Delta u=|u|^p, u_t-\Delta u=|u|^p\ ,
\] see, e.g., \cite{LW20, LZ, Fujita66}.

Despite of some partial results, particularly on the blow up part, the problem of determining the critical power (as well as giving the sharp lifespan estimates) for the problem \eqref{main} is still largely open in general.

In this paper, we would like to
show that, there exists a large class of the damping and potential functions of critical/long range, this conjecture is true, at least in the blow up part. At the same time, we are able to give upper bounds for the lifespan, which are expected to be sharp for the range $p\in (p_c-\de, p_c)$.

Before proceeding, we give the definition of energy solutions.
\begin{defn}\label{def1}
We say that $u$ is an energy solution of \eqref{main} on $[0,T]$
if
\begin{equation}\nonumber
u\in C([0,T],H^1(\R^n))\cap C^1([0,T],L^2(\R^n))\cap L^p(
[0,T]\times \R^n)
\end{equation}
satisfies
$\supp u(t,\cdot)\subset B_{t+R}$ and
\begin{equation}\label{weaksol}
\begin{aligned}
%\vep \int_{\R^n}g(x)\Psi(0, x)dx+\vep\int_{\R^n}D(x)f(x)\Psi(0, x)dx\\
&\int_0^T\int_{\R^n}|u|^p\Psi(t, x)dxdt
%+\vep\int_{\R^n}(g(x)+D(x)f(x))\Psi(0, x)dx\\
-\left.\int_{\R^n}(u_t(t,x)+D(x)u(t,x))\Psi(t, x)dx\right|_{t=0}^T\\
=&-\int_0^T\int_{\R^n}u_t(t, x)\Psi_t(t, x)dxdt+\int_0^T\int_{\R^n}\nabla u(t, x)\cdot\nabla\Psi(t, x)dxdt\\
&-\int_0^T\int_{\R^n}D(x)u(t, x)\Psi_t(t, x)dxdt+\int_0^T\int_{\R^n}V(x)u(t, x)\Psi(t, x)dxdt
\end{aligned}
\end{equation}
for any $\Psi(t, x)\in (C^0_t H_{loc}^1\cap C^1_t L_{loc}^2)([0, T]\times \R^n)$. When $n=2$, we additionally suppose
$\Psi, V \Psi, D\Psi_t\in L^{1/(1-\de_0)}_{loc}([0, T]\times \R^n)$, and
$D(x)\Psi(0, x)\in L^{ 1/(1-\de_0)}_{loc}$, for some $\de_0>0$, which ensures the integrals are well-defined. The supremum of all such time of existence,  $T$, is called to be the
lifespan to the problem \eqref{main}, denoted by $T_\vep$.
\end{defn}

%Let us divide the general situations into certain
Before presenting our main results, let us first give a brief review of the history, in a broader context. \\

{\bf (\uppercase\expandafter{\romannumeral1}) Scattering damping $D=\mathcal{O}((1+|x|)^{-\be})$, $V=0$}\\

When there are no damping and potential, this problem is related to the so-called Strauss conjecture, for which the critical power is given by $p_S(n)$, which is the positive root of the quadratic equation
\begin{equation}
\label{quadratic}
\gamma(p,n):=2+(n+1)p-(n-1)p^2=0\ ,
\end{equation}
when $n>1$.
See \cite{John79, Gla81-g,Zhou95,LS96,GLS97,LaZ1}
for global results and
\cite{John79,Gla81-b,Sch85,Sid84,YZ06, Zhou07} for blow up results
(including the critical case $p=p_S(n)$).

When there is no potential term, this problem has been widely investigated with the typical damping $D=\mu (1+|x|)^{-\be}$
\begin{equation}
\label{dampedsemi}
\left\{
\begin{aligned}
& u_{tt} - \Delta u +\mu (1+|x|)^{-\be} u_t=|u|^p, (t, x) \in [0, T)\times\R^n, \\
& u(x,0)= f(x), \quad u_t(x,0)= g(x).
\end{aligned}
\right.
\end{equation}
The asymptotic behavior of the solution to the corresponding linear damped equations has been comprehensively studied, in view of the works of \cite{IK2, ITY13, Moc76, RTY1, RTY10, RTY3, TY09, Y.Wak14}, we have the following results
\begin{center}
\begin{tabular}{|c|c|c|}
\hline
$\beta\in (-\infty, 1)$ & effective &
\begin{tabular}{c}
solution behaves like\\
that of heat equation
\end{tabular}
\\
\hline
$\beta=1$ &
\begin{tabular}{c}
scaling invariant\\
weak damping
\end{tabular} &
\begin{tabular}{c}
the asymptotic behavior\\depends on $\mu$
\end{tabular}
\\
 \hline
$\beta\in(1,\infty)$ & scattering &
\begin{tabular}{c}
solution behaves like
that\\ of wave equation without damping
\end{tabular}\\
\hline
\end{tabular}
\end{center}
Turning to the nonlinear problem \eqref{dampedsemi}, the critical power depends on the value of $\be$ and $\mu$. For $\be\in [0, 1)$, Ikehata, Todorova and Yordanov \cite{ITY09} showed that the critical power of \eqref{dampedsemi}  is the shifted Fujita exponent $p_{c}(n)=p_G(n-\be+1)=1+\frac{2}{n-\be}$. Nishihara \cite{Nishi10} studied the same damping case but with absorbed semilinear term $|u|^{p-1}u$ and proved the diffusion phenomena. For the blow-up solution, Ikeda-Sobajima \cite{IS19} gave the sharp upper bound of lifespan for the effective case $\be<1$ via the test function method, which was developed from Mitidieri-Pokhozhaev \cite{MP01}. Recently, Nishihara, Sobajima and Wakasugi \cite{N2} verified that the critical power is still $1+\frac{2}{n-\be}$ when $\be<0$.
For the critical case $\be=1$ with $\mu\ge n$,
 Li \cite{LiXin} obtained the blow-up result when $p\le p_G(n)$.

Turning to the scattering case $\be>1$, as we have discussed, it is natural to expect that the critical power is exactly the same as that of the Strauss conjecture, i.e., $p_c=p_S(n)$, see also the introduction in page $2$ in Ikehata-Todorova-Yordanov \cite{ITY} and conjecture (iii) in page 4 in Nishihara-Sobajima-Wakasugi \cite{N2}. This conjecture has been verified at least for the blow-up part when $\be>2$ and $n\ge 3$ in Lai-Tu \cite{LT20}, based on a key observation that the test function $e^{-t}\phi_1(x)$ satisfies the dual of the corresponding linear equation, where $\phi_1(x)=\int_{\Sp^{n-1}} e^{x\cdot\om}d\om$ is the one from Yordanov-Zhang \cite{YZ05}.
On the other hand, when $n=3, 4$, Metcalfe-Wang \cite{MW} obtained the global existence when $p>p_S(n)$ and $\beta>1$ with sufficiently small $|\mu|$.

Our first main result verifies the blow up part of the conjecture for the scattering damping, which, together with \cite{MW}, shows that the critical power is $p_S(n)$, at least for small scattering damping function, when
$n=3, 4$. Moreover, we improve the lifespan estimates in  \cite{LT20} for $p\le n/(n-1)$.
\begin{thm}
\label{thmStrauss}
Let $n\geq 2$. Consider the Cauchy problem \eqref{main} with $V(x)=0$. Suppose $D(x)\in C(\R^n)\cap C^{\de}(B_\de)$ for some $\de>0$ and $0\leq D(x)\leq \mu (1+|x|)^{-\be}$ with $\be>1$ and $\mu\geq 0$.
Then
for any $1<p\le p_S(n)$,
any energy solutions for  nontrivial, nonnegative, compactly supported data will blow up  in  finite time. In addition, there exist positive constants $C,\e_0$
%=\e_0(f,g,n,p,\mu, \be)>0$
such that the lifespan $T_\varepsilon$ satisfies
\begin{equation}
\label{lifespan1b}
T_\varepsilon\le
 \left\{
 \begin{array}{ll}
 C\e^{-\frac{2p(p-1)}{\gamma_0(p)}} & \mbox{for}\ 1<p\le \frac{n}{n-1},\\
 C\varepsilon^{-\frac{2p(p-1)}{\gamma(p, n)}} & \mbox{for}\ \frac{n}{n-1}\le p<p_S(n),\\
 \exp\left(C\e^{-p(p-1)}\right)& \mbox{for}\ p=p_S(n)\\
 \end{array}
 \right.
\end{equation}
for any $0<\e\le\e_0$, where 
\begin{equation}\label{gamma0}
\gamma_0(p)= -(n-1)p(p-1)+2n(p-1)+2.
\end{equation}
In addition,
the results for $p<p_S(n)$ apply also for general (short range) damping function ($D(x)\in L^n\cap L^\infty(\R^n)$ without the sign condition).
Here and in what follows, $C$ denotes a positive constant independent of $\vep$ and may change from line to line.
\end{thm}

{\bf (\uppercase\expandafter{\romannumeral2}) Critical damping $D=\mathcal{O}((1+|x|)^{-1})$ with short range potential}\\

Concerning the potential term $V$,
when it is of short range and $D=0$,  as we have discussed, it is expected that
it will not affect the critical power $p_S$. Actually, when $D=0$ and $0\le V(x)\le\frac{\mu}{1+|x|^\beta}$ with $\beta>2$ and $\mu\geq 0$, Yordanov-Zhang \cite{YZ05} proved blow up result for $1<p<p_S(n)$ when $n\geq 3$.

Our second result addresses the problem
with
possibly critical damping $D=\mathcal{O}((1+|x|)^{-1})$, together with short range potential $V$. Here, the potential $V$
 is said to be of short range, if we have %$V_\infty(r)=r^2 V(r)$ satisfies
 $r V(r)\in L^1([1,\infty), dr)$. 

\begin{thm}
\label{regular}
Let $n\geq 2$. Suppose that the coefficients $V(x),~D(x)$ satisfy:\\
{\bf 1.}~~~$V(x),\ D(x)\ \in C(\R^n)\cap C^{\delta}(B_{\delta})$ for  some $\delta>0$;\\
{\bf 2.}~~~$V(r)\ge 0$ ($V$ is nontrivial for $n=2$),\
$D(r)+V(r)\ge -1$;\\
{\bf 3.}~~~$rV(r)\in L^1([1,\infty), dr)$;\\
{\bf 4.}~~~$rD(r)= d_\infty +r D_\infty(r)$,  $D_\infty(r)\in L^1([1, \infty), dr)$,
  for some $d_\infty\in \R$, $r D\in L^\infty$.\\
Then for any $1<p<\max(p_G(n), p_S(n+d_\infty))$, there are no any global energy solutions $u$ for \eqref{main}
with $f=0$, nontrivial, nonnegative and compactly supported $g$.
Moreover, there exist constants $C, \vep_0>0$
such that $T_\vep$ is bounded from above by%has to satisfy
\begin{equation}
\label{lifespan1a}
%T_\vep\le
 \left\{
 \begin{array}{ll}
 C\vep^{-(p-1)}
%\triangleq L_1
 &
1<p<\frac{n}{n-1}, d_\infty> (n-1)(\frac 2 p-1)
,\\
%\left\{\begin{array}{ ll}
 C\vep^{-(p-1)}\left(\ln \vep^{-1}\right)^{(p-1)\max(4-n, 1)}
%\triangleq  L_2
 %      &   n\ge 3 \\
% C\vep^{-1}\left(\ln \vep^{-1}\right)^{2}      &   n=2\end{array}\right.
 &
p=\frac{n}{n-1}, d_\infty> (n-1)(\frac 2 p-1)
,\\
 C \left(\varepsilon^{-1}
(\ln \vep^{-1})^{\max(3-n, 0)}\right) ^{\frac{p-1}{(n+1)-(n-1)p}}%\triangleq L_G
&  \frac{n}{n-1}< p<p_G(n),
 d_\infty> n-1-\frac{2}{p},\\
 C\vep^{-\frac{2p(p-1)}{\gamma_1(p)}}
 \left(\ln \vep^{-1}\right)^{-\frac{2(p-1)}{\gamma_1(p)}\max(3-n, 0)}
  %\triangleq  L_4
  &  1<p<\frac{n}{n-1},
 d_\infty\le (n-1)(\frac 2 p-1),
 \\
 C\vep^{-\f{2p(p-1)}{\gamma(p,n+{d_\infty})}}(\ln \vep^{-1})^{\frac{2(p-1)}{\gamma(p,n+{d_\infty})}} %\triangleq  L_5
&p=\frac{n}{n-1}, d_\infty\le (n-1)(\frac 2 p-1)\\
% C\varepsilon^{-\frac{2p(p-1)}{\gamma(p, n+d_\infty)}} & \mbox{for}\ \frac{n}{n-1}< p<\infty, d_\infty\le 1-n\\
 C\varepsilon^{-\frac{2p(p-1)}{\gamma(p, n+d_\infty)}}%\triangleq L_S
 & \frac{n}{n-1}
< p<p_S(n+d_\infty),
d_\infty\le n-1-\frac{2}{p}
  \end{array}
 \right.
\end{equation}
for $0<\vep\le\vep_0$, where
\begin{equation}\label{gamma1}
\gamma_1(p)= -(n+d_\infty-1)p(p-1)+2n(p-1)+2.
\end{equation}
\end{thm}

See Figure \ref{fig:fig-1} for the region of $(d_\infty, p)$ where we have blow up results. Here $L_{j}$ ($j=1, 2, \cdots, 6$) denote the $j$-th lifespan, with $L_S=L_6$ and $L_G=L_3$.

\begin{figure}\label{fig:fig-1}
 \centering
  \includegraphics[width=0.5\textwidth]{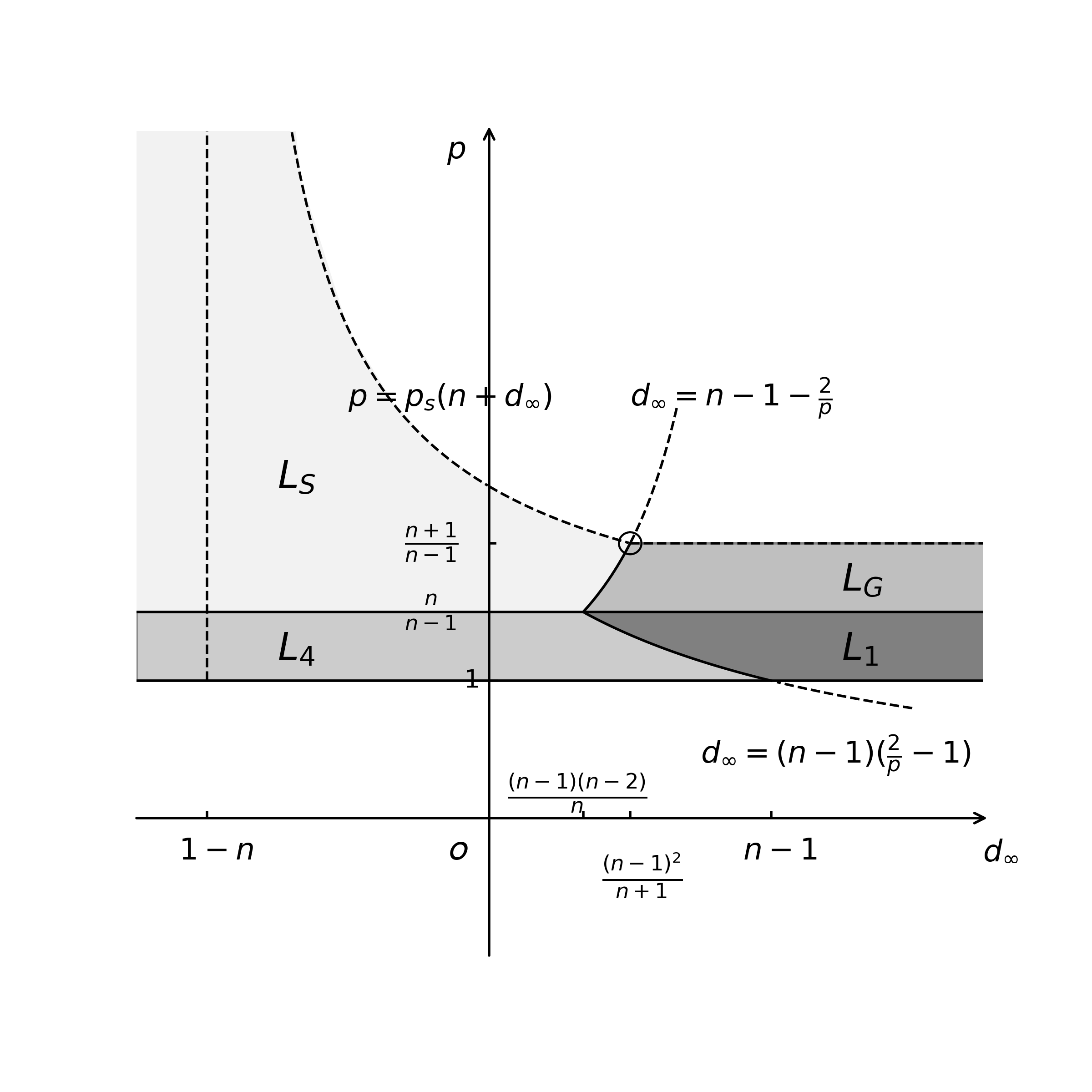}
\caption{Theorem \ref{regular}: critical powers and lifespan estimates}
\end{figure}

As mentioned above, when $D=0$, Yordanov-Zhang \cite{YZ05} has shown blow-up result for $1<p<p_S(n)$ and $n\ge 3$. The restriction for $n\ge 3$ comes from the proof of the existence and asymptotic behavior of test functions. By using ODE and elliptic theory, we could include the case $n=2$, in the case of the
radial nontrivial potential.
%generalize the corresponding result to the $2$ dimension case.
%\end{rem}

Moreover,
when $d_\infty\le 1-n$, we have $p_S(n+d_\infty)=\infty$ and so there are no global   solutions for any $p\in (1,\infty)$. A specific example
could be $D=\frac{d_\infty}{d_\infty+r}$ and $V=V_\infty(1+r)^{-2}$.

\begin{rem}
Through scaling, we see that the technical condition $1+D(r)+V(r)\ge 0$ could be replaced by the slightly general condition
$\lambda_0^2+\la_0 D(r)+V(r)\ge 0$ for some $\lambda_0>0$.
\end{rem}

{\bf (\uppercase\expandafter{\romannumeral3}) Critical damping and potential $D=\mathcal{O}(|x|^{-1})$, $V=\mathcal{O}(|x|^{-2})$ near spatial infinity}\\

Finally, when both the damping and potential terms exhibit certain critical nature, it turns out that
we have the blow-up phenomenon under the shifted Strauss exponent and a shifted Glassey exponent, shifted by $\rho(v_\infty)$ or $d_\infty$.

In \cite{GKW19}, Georgiev, Kubo and Wakasa
showed that the critical power for the radial solutions is the shifted Strauss exponent $p=p_S(3+2)$, for a special case in $\R^3$, with damping and potential coefficients satisfying the relation
$$V(r)=-D'(r)/2+D^2(r)/4,$$
where $D(r)$ is a positive decreasing function in $C([0, \infty))\cap C^1(0, \infty)$ satisfying $D(r)= 2/r$ for $r\ge r_0>0$.

In the case of the problem with sample scale-invariant ``critical'' damping and potential,
$$D(x)=\frac{d_\infty}{|x|}, V(x)=\frac{v_\infty}{|x|^2}\ ,$$
with %constants $A, B$ satisfying
%\begin{equation}\label{AB}
%\begin{aligned}
$0\le d_\infty<n-1+2\rho(v_\infty)$, $v_\infty>-(n-2)^2/4$,
%\end{aligned}\end{equation}
thanks to the specific structure of the damping and potential terms,
Dai, Kubo and Sobajima \cite{DKS21} were able to construct explicit test functions by using hypergeometric functions and obtain the upper bound of the lifespan for $$\frac{n+\rho(v_\infty)}{n+\rho(v_\infty)-1}<p\le p_c\ .$$
See also Ikeda-Sobajima \cite{IS2} for the prior blow-up results for $\frac{n}{n-1}<p\le p_S(n+d_\infty)$,
 when $v_\infty=0$, $n\geq 3$ and $0\le d_\infty<(n-1)^2/(n+1)$.

As we can see, the above results heavily depends on the specific structure of the damping and potential terms. Our next theorem addresses on the problem with a general class of damping and potential terms, which
exhibit certain critical nature.

\begin{thm}
\label{general}
Consider \eqref{main} with $n\geq 2$.
Assume that the coefficients $V(x),~D(x)$ satisfy:

{\bf 1.} ~~~ $V(r), D(r)\in C(\R_{+})$, $V(r)\ge 0$, $D+V\ge -1$;

{\bf 2.} For the part near spatial infinity, $r>1$,

~~~ $\displaystyle r^2 V(r)=v_\infty+ r V_\infty(r)$,  $V_\infty(r)\in L^1_{r>1}$, $v_\infty\ge 0$ (with $v_\infty>0$~ for~ $n=2$);

~~~ $\displaystyle rD(r) =  d_\infty + r D_\infty(r),\ D_\infty(r) \in L^1_{r>1},\ rD_\infty\in L^\infty_{r>1}$ with $d_\infty\in\R$;

{\bf 3.}~~~ For $r\in (0,1]$,  $D=\mathcal{O}(r^{\theta-2})$ for some $\theta\in [0, 2]$, and

~~~ $\displaystyle r^2V(r)= v_0+r V_0(r),\ V_0(r)\in L^1_{loc}$ for $n\ge 3$ or $v_0>0$;

~~~ $\displaystyle r^2 D(r)= d_0 +r D_0(r),\ D_0(r)\in L_{loc}^1$ for $n\ge 3$ or $v_0+ d_0>0$.\\
For $n=2$ and the endpoint case, instead of the assumption {\bf 3}, we assume the analytic conditions for some $\de>0$:

{\bf 3'.}~~~ ~~~ $\displaystyle r^2V(r)=\sum_{j\ge 1}b_j r^j$ for $r\in (0,\de)$ if $v_0=0$;

~~~ $\displaystyle  r^2(V+ D)=\sum_{j\ge 1}c_j r^j$ for $r\in (0,\de)$ if $v_0+d_0=0$.

Suppose that \[p_0:=\f{n+\rho(v_0)}{n+\rho(v_0+\min(0, d_0))+\theta-2}<p_c=\max(p_S(n+d_\infty), p_G(n+\rho(v_\infty)),\]
then
for any $p_0<p<p_c$, any energy solutions $u$ for \eqref{main},
with
$f=0$, nontrivial, nonnegative and compactly supported $g$,
%nontrivial, nonnegative and compactly supported data
will blow up in finite time.
Moreover, there exist constants $C, \vep_0>0$
such that $T_\vep$ has to satisfy
\begin{equation}
\label{lifespan1}
T_\vep \le
 \left\{
 \begin{array}{ll}
 C\vep^{-(p-1)}& \mbox{if}\  p\in(p_2, p_3)\neq \emptyset \\
 C\vep^{-(p-1)}\left(\ln \vep^{-1}\right)^{p-1}& \mbox{if}\ p=p_3>p_2\\
  C\vep^{-\f{p-1}{n+\rho(v_\infty)+1-(n+\rho(v_\infty)-1)p}}& \mbox{if}\ p\in (\max(p_2, p_3), p_G(n+\rho(v_\infty)))\neq \emptyset\\
% C\vep^{-\f{p-1}{1-(n+\rho(v_\infty)-1)\max(0, p-p_3)}} \left(\ln \vep^{-1}\right)^{(p-1)\de_{p, p_3}} & \mbox{if}\ p\in (p_2, p_G(n+\rho(v_\infty)))\neq \emptyset\\
 C\vep^{-\f{2p(p-1)}{\gamma_2(p)}}& \mbox{if}\
p\in(p_0, \min(p_3, p_5))\neq \emptyset
 \\
 C\vep^{-\f{2p(p-1)}{\gamma(p,n+{d_\infty})}}(\ln \vep^{-1})^{\frac{2(p-1)}{\gamma(p,n+{d_\infty})}} & \mbox{if}\ p=p_3\in (p_0,  p_S(n+{d_\infty}))\neq \emptyset\\
 C\vep^{-\f{2p(p-1)}{\gamma(p,n+{d_\infty})}} & \mbox{if}\ p\in (\max(p_0, p_3),  p_S(n+{d_\infty}))\neq \emptyset\\
 C\vep^{-\f{2(p-1)}{(n+d_\infty+1)-(n+d_\infty-1)p}}& \mbox{if}\ p\in
  (p_4, p_0]\cap (p_4, p_G(n+d_\infty))
\neq \emptyset
 \end{array}
 \right.
\end{equation}
for $0<\vep\le\vep_0$.
Here
%2+2(n+\rho(v_\infty))(p-1) -(n+d_\infty-1)p(p-1),$$ and
%$$\gamma_2(p):= -(n{+d_\infty}-1)p^2+(3n+2\rho(v_\infty){+d_\infty}-1)p+(2-2n-2\rho(v_\infty)),$$ and
$p_i$ are defined as follows:
$$
p_0=\max(p_1, p_2), \ p_1:=
\f{n+\rho(v_0)}{n+\rho(v_0+d_0)+\theta-2},\
p_2:=
\f{n+\rho(v_0)}{n+\rho(v_0)+\theta-2}$$
\[p_3:=\f{n+\rho(v_\infty)}{n+\rho(v_\infty)-1}
\ ,\
p_4:=\f{n+\rho(v_0+d_0)}{n+\rho(v_0+d_0)+\theta-2}.\]
$$ p_5=\left\{\begin{array}{ll}
%\frac{n+1+\sqrt{(n+1)^2+8(n-1)}}{2(n-1)}
\frac{3n+d_\infty+2\rho(v_\infty)-1+\sqrt{(3n+d_\infty+2\rho(v_\infty)-1)^2-8(n+d_\infty-1)(n+\rho(v_\infty)-1)}}{2(n+d_\infty-1)}
& n+d_\infty>1\\
\infty & n+d_\infty\le 1\ ,
\end{array}\right.$$
and
\beeq\label{eq-gamma_2}
\gamma_2(p):= 2(n+\rho(v_\infty)-1)(p-p_3)+\ga(p, n+d_\infty)\ ,\eneq
where $\ga$ is given in \eqref{quadratic} and $p_5$ is the positive root of $\ga_2(p)=0$ when $n+d_\infty>1$.
\end{thm}
\begin{rem}\label{rem-6}
We observe that, in our statement, besides the expected upper bound of the blow up range, we have certain lower bound, which depends on the local behavior of $V(x)$ and $D(x)$ near the origin, as well as $v_\infty$ for $p_3$. Heuristically, we expect that the lower bounds for the blow-up range are just technical conditions. In other words, we conjecture that we still have blow up results for
any $p\in (1, p_c)$.
%any $p\in (1, \max(p_S(n+d_\infty), p_G(n+\rho(v_\infty))))$.
\end{rem}

\begin{rem}
When
$n=3$, $v_0=d_0=0$, $v_\infty=d_\infty=2$,
and $\theta=2$, we have $p_1=p_2=1$, $p_3=4/3$, and thus recover the blow-up result  in \cite{GKW19} for $1<p<p_S(5)$. 
 When $d_0=0$, $d_\infty=A$, $v_0=v_\infty=B$ and $\theta=1$, our result recovers and generalizes the subcritical results in \cite{DKS21}. The blow-up and lifespan estimate for the ``critical'' power are lost in our result, which seems to be much more delicate to obtain, in our general setting for the damping and potential functions. One of the unexpected features is that we could
obtain blow up results for
highly singular damping term near the origin. For example, for any
$d_0, v_0, d_\infty, v_\infty>0$ such that $p_0<p_c$, the problem with
$$D= \frac{d_0}{r^2}+\frac{d_\infty}{ r},\ V=\frac{v_0}{ r^{2}}+ \frac{2(v_\infty-v_0)\arctan r}{\pi r^2}\ ,$$
could not admit global solutions in general, for $p\in (p_0, p_c)$, despite the strong damping effect near the origin.
% This phenomenon is
\end{rem}

Concerning the comparison of the exponents $p_i$, $p_S$ and $p_G$, they depend on the the damping, potential and dimension.
For example, we observe from
\eqref{eq-gamma_2} and
\eqref{quadratic} for $\ga_2$ and $\gamma(p,n)$, whose positive roots give $p_5$ and $p_S$, as well as the obvious relation $p_3<p_G(n+\rho(v_\infty))$, that
we always have $$\min(p_3, p_5)<p_c\ .$$
As $n+\rho(d_\infty)-1>0$, we have $\gamma_2(p)<\gamma(p, n+d_\infty)$ for $p<p_3$, which shows that the current lifespan estimate for $p_0< p< \min(p_3, p_5)$ is weaker than the standard one.
Also, it is clear that $p_1<p_2$ iff $d_0>0, \theta<2$, which is also equivalent to $p_4<p_2$.

In relation with Remark \ref{rem-6}, we would like to
show blow up results for any $p\in (1, p_c)$.
 It is clear that when $\theta=2$ (which ensures that $d_0=0$), we have
 $p_0=1$ and so is the blow up results for $p\in (1, p_c)$. Moreover, by examining the proof of Theorem \ref{general},
in particular the estimates \eqref{F0} and \eqref{step1},
  we find that the technical condition $p_3$ for the lifespan estimates could also be avoided if the damping term vanishes $D=0$ near spatial infinity.
That is, when $D(r)\in C_c([0,\infty))$ (so that $d_0=d_\infty=0$ and $\theta=2$), we have
$$
F_0=
\int_{T/2}^T\int_{B_{t+R}}
T^{-2p'}\phi_0
dxdt \les T^{-2p'+\rho(v_\infty)+n+1}\ ,$$
and so is the following
\begin{cor}\label{thm-VanDamp}
Under the same assumptions as in Theorem \ref{general}, with an additional assumption that $D(|x|)\in C_c([0,\infty))$.
Then we have
blow up results for any $1<p<p_c=\max(p_G(n+\rho(v_\infty)), p_S(n))$. In addition, we have, for
some constants $C, \vep_0>0$,
\begin{equation}
\label{eq-lifespan2}
T_\vep \le
 \left\{
 \begin{array}{ll}
L_S= C\vep^{-\f{2p(p-1)}{\gamma(p,n)}} & \mbox{if}\ p\in (1,  p_S(n))\\
L_G= C\vep^{-\f{p-1}{n+\rho(v_\infty)+1-(n+\rho(v_\infty)-1)p}}& \mbox{if}\ p\in (1, p_G(n+\rho(v_\infty)))
 \end{array}
 \right.
\end{equation}
for any $0<\vep\le\vep_0$.
\end{cor}

See Figure \ref{2D-bu}
 for the region of $(v_\infty, p)$ where we have blow up results for $n=2$. As we assume $v_\infty\ge 0$, we always have $p_S(n)> p_G(n+\rho(v_\infty))$ and the second upper bound $L_G$ is effective (i.e., $L_G\le L_S$ for $\vep\ll 1$ if and only if
 $(2\rho(v_\infty)+n-1)p\le 2$, which is nonempty only if $n\le 2$.

 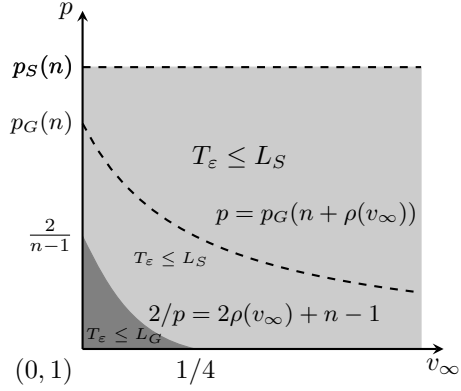
\begin{figure}%[h]
\centering
\begin{tikzpicture} [scale=1.5]
\filldraw[black!20!white] (0,2)--(0,3.5)--(3,3.5)--(3,1)
--(1,1) to[out=167,in=-64] (0,2);
\filldraw[black!50!white] (0,2)--(0,1)--(1,1) to[out=167,in=-64] (0,2);
\node[right] at (0.5,1.3) {\small $2/p=2\rho(v_\infty)+n-1$};
\node[left] at (0.8,1.12) {\tiny $T_\vep \le L_G$};
\draw[thick,-stealth] (0,1)--(0,4) node[left]{$p$};
\draw[thick,-stealth] (0,1)--(3.2,1) node[below]{$v_\infty$};
\node[below] at (1,1) {$1/4$}; 
\node[left] at (0,2) {$\frac{2}{n-1}$};
\node[left] at (0,3) {\small $p_G(n)$};
\node[left] at (0,3.5) {\small $p_S(n)$};
%\node[above] at (1.1,3.5) {\small $p=p_S(n)$};
\draw[thick, dashed] (0,3.5)--(3, 3.5);
\node[left] at (0,3.5) {\small $p_S(n)$};

\draw[thick, dashed, domain=0:3] plot (\x, {{1+2/((\x)+1)}}) ;
\node[right] at (1.1,2.2) {\small $p=p_G(n+\rho(v_\infty))$};
\node[left] at (1.9,2.7) {$T_\vep \le L_S$};
\node[left] at (1.2,1.8) {\tiny $T_\vep \le L_S$};
\node[below left] at (0,1) {$(0,1)$};

\end{tikzpicture}
\caption{Corollary \ref{thm-VanDamp}: lifespan estimates for $n=2$}
\label{2D-bu}
\end{figure}

%\begin{figure*}%%\sidecaption
 %\begin{minipage}{4cm}
%\centering
 % \includegraphics[width=1\textwidth]{1.png}
  %%\subcaption{fig1}
  %\end{minipage}
 % \begin{minipage}{4cm}
  %\centering
  %\includegraphics[width=1\textwidth]{2.png}
  %%\subcaption{fig2}
  %\end{minipage}
  %\begin{minipage}{4cm}
  %\centering
  %\includegraphics[width=1\textwidth]{3.png}
  %%\subcaption{fig2}
  %\end{minipage}
%\caption{}
%\end{figure*}

%\marginpar{\color{blue}strategy and structure}

For the strategy of proof, we basically follow the test function method.
 The key ingredient is to construct
special positive standing wave solutions, of the form $w=e^{-\la t}\phi_\la(x)$, to the linear dual problem,
$$(\pt^2-\Delta-D\pt +V) w=0\ ,$$
with the desired asymptotic behavior. In turn, it is reduced to constructing solutions to certain elliptic ``eigenvalue" problems:
\beeq\label{eq-eigen} (\la^2-\Delta+D\la +V) \phi_\la=0\ .\eneq
Concerning the subcritical blow up results, $p<p_c$, it suffices to construct solutions for \eqref{eq-eigen} with $\la=0$ and some $\la_0>0$. While for the critical case, we will also need to find solutions with uniform estimates with respect to all $\la\in (0,\la_0)$.

\subsubsection*{Outline} Our paper is organized as follows. In Section \ref{sec:1stTest},
we present
the existence results of special solutions for the elliptic ``eigenvalue" problems \eqref{eq-eigen}, with certain asymptotic behavior,
 by applying elliptic and ODE theory.
 These solutions play a key role in constructing the test functions and the proof of blow up results.
 The proof of the existence results are given in Section
 \ref{sec-pf-test}.
Equipped with the eigenfunctions, the test function method is implemented in Section \ref{sec:AE3}, to give the proof of Theorem \ref{general}.
Then, in Section \ref{sec:regular}, we give the proof of Theorem \ref{regular}, when $n\ge 2$ and the potential is of short range. In essence, with the help of
Lemma \ref{lem8} and \ref{lem2}, Theorem \ref{regular} could be viewed as a corollary of  Theorem \ref{general} when $n\ge 3$.
%The two dimensional case is a little different
At last, in Section \ref{sec:thmStrauss}, we
present the required test function for the critical case, as well as the upper bound and lower bound estimates. Equipped with the test function, a relatively routine argument (see, e.g.,
 \cite{LT20} or \cite{LLWW}) will yield Theorem \ref{thmStrauss}.

\subsubsection*{Notations.}
We close this section by listing the notation. Let $\langle x\rangle = \sqrt{1 + |x|^{2}}$ for $x\in \R^n$.
We will also use $A\les B$ to stand for $A\le C B$ where the constant $C$ may change from line to line.

\section{Solutions to elliptic equations}\label{sec:1stTest}
In this section, we present the existence results of special solutions for some elliptic ``eigenvalue" problems, with certain asymptotic behavior, which would be used to construct test functions to derive the expected lifespan estimates.
We will consider two types of elliptic equations.
%Let $D(x)=D(|x|)$, $V(x)=V(|x|)$ be two spherically symmetric functions, then we have the following.

\subsection{Eigenfunction for $V-\Delta$}
At first, we consider the ``zero-eigenfunction" for the positive elliptic operator $V-\Delta$.
\begin{lem}
\label{lem1}

Suppose $0\leq V\in C(\R^n\backslash\{0\})$ and
$$
 r^2 V(r)=v_\infty+ r V_\infty(r)
 =v_0+rV_0(r)
 \ ,  V_\infty(r)\in L^1_{r>1}\ , V_0(r)\in L^1_{r<1}\ ,$$
Then if $n\geq 3$ or $v_0, v_\infty>0$, there exists a solution $\phi_0\in H_{loc}^1(\R^n)\cap W_{loc}^{2, 1+}(\R^n)$ of
\beeq
\label{s31}
\Delta \phi_0=V\phi_0, \ x\in\R^n,\  \ n\geq 2,\
\eneq
 satisfying
\beeq
\label{lem1-1}
\begin{split}
\phi_0\simeq
\begin{cases}
r^{\rho(v_0)},\ r\leq 1,\\
r^{\rho(v_\infty)},\ r\geq 1\ .
\end{cases}
\end{split}
\eneq
In addition, for $n=2$ with $v_0=0$, if we assume $r^2V$ is analytic in $(0,\de)$ for some $\de>0$, i.e., $$
r^2V(r)=\sum^{\infty}_{j=1}b_j r^{j}, \  0<r<\de\ ,  $$
then the same result holds.
\end{lem}

When $V$ is H\"older continuous near the origin, the regularity of the solutions could be improved.
%In the following, we show that there exist $C^2$ solutions of \eqref{s31}
\begin{lem}
\label{lem8}
Suppose $0\leq V\in C(\R^n)$,
$V=V(r)=\mathcal{O}(\langle r\rangle^{-\be})$ for some $\be>2$,
 and $V\in C^\de(B_\de)$ for some $\de>0$.  In addition, we assume $V$ is nontrivial when $n=2$. Then there exists a $C^2$ solution of \eqref{s31} satisfying
\beeq
\begin{split}
\phi_0\simeq
\begin{cases}
\ln(r+2),\  n=2,\\
1,\  \ n\geq 3.\
\end{cases}
\end{split}
 \eneq
Moreover, when $n=2$ and $V\equiv 0$, it is clear that $\phi_0=1$ is a solution
\eqref{s31}.
\end{lem}
%\begin{rem}
When $n\geq 3$, $V$ is locally H\"{o}lder continuous (not necessarily radial) and $0\leq V\leq \frac{C}{1+|x|^{2+\de}}$ with $C,\de>0$, Lemma \ref{lem8} was known from Yordanov-Zhang \cite{YZ05}.
%Here we give a simple proof when $V$ is radial for $C^2$ solutions.
%\end{rem}

\subsection{``Eigenfunction" for $V+D-\Delta$ with negative ``eigenvalue"}
In this subsection, we consider the
``eigenfunction" for $V+D-\Delta$ with negative ``eigenvalue":
\beeq
\label{s32}
\Delta\phi_{\la_0}=(\la_0^2+\la_0 D+V)\phi_{\la_0}, \ \la_0>0,\  \ x\in\R^n.\
\eneq

\begin{lem}
\label{lem4}
Let $V=V(r),  D=D(r)\in C(0, \infty)$. Suppose $V\ge 0$, $D\ge -\la_0-\frac{V}{\la_0}$ for some $\la_0>0$,  and
$$
D(r)=\frac{d_0}{r^2}+\frac{1}{r}D_0(r),\  V(r)=\frac{v_0}{r^2}+\frac{1}{r}V_0(r),\  r\leq 1,\ $$
with $D_0(r)\in L^1_{r<1}$, $V_0(r)\in L^1_{r<1}$.
 In addition, we assume for some $d_\infty\in \R$,
$$V(r)\in L^{1}_{r>1},\ D(r)=\frac{d_\infty}{r}+D_\infty(r),\  D_\infty(r)\in L^{1}_{r>1}.\ $$
 Then if $n\geq 3$ or $v_0+\la_0d_0>0$, there exists a $H_{loc}^1(\R^n)\cap W_{loc}^{2, 1+}(\R^n)$ solution of \eqref{s32} satisfying
\beeq
\label{lem1-2}
\begin{split}
\phi_{\la_0}\sim
\begin{cases}
r^{\rho(v_0+\la_0d_0)},\ r\leq 1,\\
r^{-\frac{n-1-d_\infty}{2}}e^{\la_0 r},\ r\geq 1.
\end{cases}
\end{split}
\eneq
In addition, we have the same result for $n=2$ with $v_0+\la_0d_0=0$, when $r^2(V+\la_0 D)$ is analytic near $0$:
$$r^2(V+\la_0D)=\sum^{\infty}_{j=1}c_j r^j\ , 0<r<\de\ .$$
\end{lem}

\begin{lem}
\label{lem2}
Let $V=V(r),  D=D(r)\in C(\R^n)\cap C^\de(B_\de)$ for some $\de>0$. Suppose $V\ge 0$, $D\geq -\la_0-\frac{V}{\la_0}$,
and also that for some $d_\infty\in \R$, $R>1$ and $r\geq R$,
we have $$V(r)\in L^{1}([R, \infty))\ ,\ D(r)=\frac{d_\infty}{r}+D_\infty(r), \ D_\infty(r)\in L^{1}([R, \infty))\ .$$
 Then there exists a $C^2$ solution of \eqref{s32} satisfying
\beeq
\phi_{\la_0}\simeq
\langle r\rangle^{-\frac{n-1-d_\infty}{2}}e^{\la_0r}\ .
\eneq
\end{lem}

\subsection{Eigenvalue problem with parameters}
To handle the critical problem, we will need
to construct a class of (unbounded) positive solutions for the eigenvalue problem, with certain uniform  estimates for small parameters.

\begin{lem}
\label{elp}
Let $n\geq 2$, $\be>1$, $\mu\geq 0$. Suppose $D(x)\in C(\R^n)\cap C^{\de}(B_\de)$ for some $\de>0$ and $0\leq D(x)\leq \frac{\mu}{(1+|x|)^{\be}}$.
Then there exists $c_1\in (0, 1)$ such that for any
$0<\la\le 1$, there is a $C^2$ solution of
\beeq
\label{509}
\Delta\phi_\la-\la D(x)\phi_\la=\lambda^{2}\phi_\la
\eneq
satisfying
\beeq
\label{510}
c_{1} \< \la |x|\>^{-\frac{n-1}{2}}e^{\la |x|}%\< \la r\>^{-\frac{n-1}{2}}e^{\la \int^{r}_{0}K(\tau)d\tau}
<\phi_\la(x) < c_1^{-1}\< \la |x|\>^{-\frac{n-1}{2}}e^{\la |x|} \ .
\eneq
\end{lem}

\section{Proof of Lemmas \ref{lem1}-\ref{elp}}\label{sec-pf-test}
In this section, we present the proof of Lemmas \ref{lem1}-\ref{elp}. by applying elliptic and ODE theory.

\subsection{Proof of Lemma \ref{lem1}}
We will find a radial solution $\phi_0(x)=\phi_0(|x|)=\phi_0(r)$.

At first,
for the region $0<r\leq 1$, we observe that the equation in $r$ is of the Euler type:
$$\Delta \phi_0=(\pa^{2}_r+\frac{n-1}{r}\pa_r)\phi_0=V\phi_0\ ,$$
for which it is natural to introduce
 a new variable $t$ with $r=e^{-t}$. Then $f(t)=\phi_0(e^{-t})$ satisfies
$$
\pa^{2}_t f-(n-2)\pa_t f=(v_0+e^{-t}V_0(e^{-t}))f\ , t\geq 0\ .
$$
Let $Y(t)=(Y_{1}(t), Y_{2}(t))^{T}$ with $Y_{1}=f$, $Y_{2}=\pa_t f$, then we have
$$Y'=(A+B(t))Y,
A=\left(\begin{array}{cc}0 & 1 \\ v_0 & n-2\end{array}\right),
B(t)=\left(\begin{array}{cc}0 & 0 \\ e^{-t}V_0(e^{-t}) & 0\end{array}\right)\ .$$
As $V_0(r)\in L^1_{r<1}$, that is, $B\in L^1[0, \infty)$, we could apply the Levinson theorem (see, e.g.,
\cite[Chapter 3, Theorem 8.1]{CoLe55}) to the system. Then there exists $t_0\in [0,\infty)$ so that we have two independent solutions, which have the asymptotic form
as $t\to \infty$
$$Y_{+}(t)=
\left(\begin{array}{c}1+o(1) \\
\la_1+o(1)\end{array}\right)
e^{\la_1 t}
\ , \ Y_{-}(t)=
\left(\begin{array}{c}1+o(1) \\
\la_2+o(1)\end{array}\right)
e^{\la_2 t}\ ,
$$
where $\la_1=\sqrt{(\frac{n-2}{2})^{2}+v_0}+\frac{n-2}{2}$, $\la_2=-\sqrt{(\frac{n-2}{2})^{2}+v_0}+\frac{n-2}{2}=-\rho(v_0)$.

We choose $Y(t)=Y_-(t)$ so that
$$\phi_0(r)=f(-\ln r)=(1+o(1))
r^{\rho(v_0)}$$
as $r\to 0$.
It is easy to check that $\phi_0\in H^1(B_\de)\cap W^{2, 1+}(B_\de)$, and
$$\pa_r\phi_0=-r^{-1} \pt f|_{t=-\ln r}=
(\rho(v_0)+o(1))r^{\rho(v_0)-1}
$$ for $r\in (0, 1]$.
Based on the assumption $n\geq 3$ or $v_0>0$, we have
$\rho(v_0)>0$ and so there exists $\de_0>0$ such that
\beeq\label{eq-test-d}\pa_r\phi_0\simeq r^{\rho(v_0)-1}\ , r\in (0,\de_0)\ .\eneq

In addition, for the case $n=2$ and $v_0=0$
when $r^2V$ is analytic in $(0,\de)$, by applying the Frobenius method, there is an analytic solution
$$\phi_0=r^{\la_1}\sum^{+\infty}_{j=0}a_j r^j \sim 1, \ r< \de\ ,$$
with $a_0=1$, where $\la_1=\rho(v_0)=0$ is the root of $\la^2=0$.
If $V\equiv 0$ near $0$, then $\phi_0=1$. Otherwise, there exists $k>0$ such that $b_j=0$ for all $j<k$ and $b_k>0$, in which we have
$a_j=b_j/j^2$ for any $j\in [1,k]$. Thus once again we have, for some $\de_0>0$, $$\pa_r \phi_0\simeq (\frac{b_k}{k}+o(1))r^{k-1}\ , r\in (0,\de_0)\ .$$

Similarly, for $r>1$, let $r=e^t$, then $F(t)=\phi_0(e^t)$ satisfies
$$
\pa^{2}_t F+(n-2)\pa_t F=(v_\infty+e^t V_\infty(e^t)) F\ , t\geq 0\ .
$$
Applying the Levinson theorem again, we know that there exists $c_1, c_2$ such that
$$
\phi_0(e^t)=c_1(1+o(1))
e^{\rho(v_\infty) t}
+c_2(1+o(1))
e^{(-\rho(v_\infty)-(n-2))t},$$
$$\pt \phi_0(e^t)=c_1(\rho(v_\infty)+o(1))
e^{\rho(v_\infty) t}
+c_2(2-n-\rho(v_\infty)+o(1))
e^{(-\rho(v_\infty)-(n-2))t}
$$
as $t\to\infty$.
Notice that $\rho(v_\infty)>0$, due to the assumption $n\geq 3$ or $v_\infty>0$.

By the fundamental well-posed theory of linear ordinary differential equation, we know that $\phi_0(r)\in C^2(0, \infty)$.
We claim that $\pa_r\phi_0(r)\geq 0$ for all $r>0$. Actually, we have seen from
\eqref{eq-test-d}
 that $\phi_0>0$ and $\pa_r\phi_0 \geq 0$ for $r\in (0, \de_0)$.  Suppose, by contradiction, there exists a $r_2>\de_0$ such that $\phi'_0(r_2)<0$ with $\phi_0(r)\geq 0$ for any $r\in (0, r_2]$. Then there is a $r_1<r_2$ such that $\pa_r\phi_0(r_1)=0$.  Recall that
$$\pa_r(r^{n-1}\pa_r\phi_0)=r^{n-1}V\phi_0\ .$$
By integrating it form $r_1$ to $r_2$, we get
$$r^{n-1}_2\phi'_0(r_2)=\int^{r_2}_{r_1}\tau^{n-1}V\phi_0(\tau)d\tau\geq 0\ ,$$
which is a contradiction. Hence we
%Since we have $\phi_0(r)\in C^2(0, \infty)$ and $0\leq V\in C(0, \infty)$, by maximum principle of elliptic equation, we have $\pa_r\phi_0\geq0$ for all $r>0$,
get $c_1>0$ and
$$\phi_0\sim r^{\rho(v_\infty)},
\pa_r\phi_0\sim r^{\rho(v_\infty)-1}
 \ r \gg1 \ .$$

\subsection{Proof of Lemma \ref{lem8}}
Suppose $V$ is H\"older continuous in $\overline{B_\de}$ for some $\de>0$. Then we
consider the Dirichlet problem
\beeq
\label{617-2}
\begin{cases}
\Delta\phi_0=V\phi_0, x\in B_{\de}\\
\phi_{0}|_{\pa B_{\de}}=1
\end{cases}
\eneq
By Gilbarg-Trudinger \cite[Theorem 6.14]{GT}, there exists a unique $C^2(\overline{B_\de})$ solution, which must be radial.  For $r>0$, with $\phi_0(x)=\phi_0(r) $, \eqref{617-2} becomes an ordinary differential equation
\beeq
\label{617}
\begin{cases}
\phi''_{0}+\frac{n-1}{r}\phi'_{0}-V\phi_0=0\\
\phi_0(\de)=1, \phi'_0(\de)=C
\end{cases}
\eneq
Then by the theory of ordinary differential equation, there is a unique solution $\phi_0(r)\in C^2(0, \infty)$, which agrees with $\phi_0$ in $B_\de$. Thus we get a solution $\phi_0(x)\in C^2(\R^n)$ and we could apply strong maximum principle to get $\pa_r\phi_0\geq0$ for all $r>0$ and $\phi_0(0)>0$.

Recall that
\beeq
\label{2-d1}
\pa_r(r^{n-1}\pa_r\phi_0)=r^{n-1}V\phi_0\ .
\eneq
By integrating it from $0$ to $r$, we get
\beeq
\label{2-d}
r^{n-1}\pa_r\phi_0=\int^{r}_{0}\tau^{n-1}V\phi_0 d\tau \ \les \ \phi_0\int^{r}_{0}(1+\tau)^{n-1-\be}d\tau\ .
\eneq
Thus, as $\be>2$, if $n\geq 3$ and $r\ge 1$, we have
$$\pa_r\phi_0\ \les \
 \phi_0r^{1-n}\int^{r}_{0}(1+\tau)^{n-1-\be}d\tau\les
\phi_0 r^{-1-\de}\ ,$$
for some $\de>0$.
By Gronwall's inequality, we obtain
$\phi_0 \les  \phi_0(1)$, for all $r\ge 1$,
which yields $$\phi_0\simeq 1\ .$$

For $n=2$, since $V$ is nontrivial, then there exists some $R>0$ such that $V\neq 0$ when $R\leq r \leq 2R$. By \eqref{2-d1}, we have
$$r\pa_r\phi_0=R\pa_r\phi_0(R)+\int^{r}_{R}\tau V\phi_0 d\tau \ \ge \ \phi_0(R)R\int^{2R}_{R}V(\tau)d\tau\geq C\ , \ r\geq 2R\ ,$$
hence we get
$$\phi_0(r)\geq \phi_0(R)+\int^r_R \frac{C}{\tau}d\tau \geq C\ln \frac{r}{R}+\phi_0(R) \ \ges \ \ln (2+r),\  r\geq 2R\ .$$
On the other hand, by \eqref{2-d}, we have
$$r\pa_r\phi_0=\int^{r}_{0}\tau V\phi_0 d\tau \ \les \ \phi_0(r)\int^{\infty}_{0} (1+\tau )^{1-\be} d\tau\  \les \ \phi_0(r)\ , $$
which gives us
\beeq\label{2-d-1}\phi_0 \ \les \  1+r\ , \forall \ r\geq 0\ .\eneq
By inserting \eqref{2-d-1} into \eqref{2-d}, we have for any $\de_1\in (0, \be/2-1)$ and $r\geq 1$
\beeq
\pa_r\phi_0\les
\left\{\begin{array}{ll}
r^{-1}, & \be>3,\\
r^{2-\be+\de_1},&\ 2<\be\leq 3\ .
\end{array}\right.
\eneq
Then it easy to obtain the desired upper bound $\ln(2+r)$ for $\be>3$, while for $\be\leq 3$,  we get $\phi_0(r)\les (1+r)^{3-\be+\de_1}$ for any $r\geq 0$ which is better than \eqref{2-d-1}.
%, when $\de_1<\be-2$.

 Thus, to obtain the expected upper bound, we do the iteration, by inserting the improved upper bound $\phi_0(r)\les (1+r)^{3-\be+\de_1}=(1+r)^{k}$ into \eqref{2-d} to get, for $r>1$,
\beeq
\nonumber
\pa_r\phi_0\les
\begin{cases}
r^{-1}, \ k+2<\be\\
r^{k+1-\be+\de_2},\  k+2\geq\be,
\end{cases}
\phi_0\les
\begin{cases}
\ln(2+r), \ k+2<\be\\
(1+r)^{k+2-\be+\de_2},\  k+2\geq\be\ ,
\end{cases}
\eneq
for any $\de_2\in (0, \be/2-1)$.
For $k+2\geq \be$, we can get the improved upper bound.
% if we take $\de_2$ small enough such that $2-\be+\de_2<0$.
 By repeating (finitely) steps, we can finally obtain
$$\phi_0(r)\les \ln (2+r)\ , \forall\ r\geq 0\ .$$

\subsection{Proof of Lemma \ref{lem4}}
To start with, we record a lemma from Liu-Wang \cite{LW2019} which we will use later.
%\begin[\cite[Lemma 3.1]{LW2019}]{lem}
\begin{lem}[Lemma 3.1 in \cite{LW2019}]
\label{thm-ode}
Let $\la\in (0, \la_0]$, $\de_0\in (0, 1)$, $\ep>0$, $y_0>0$,  $K\in (\de_0, \de_0^{-1})$,
\beeq
\label{0-eq-Gest}
\|K'\|_{L^1([\ep \la_0^{-1},\infty))}\le \de_0^{-1},
\|G\|_{L^1([\ep \la^{-1},\infty))}\le \de_0^{-1} \la, \forall \la\in (0, \la_0]\ .
\eneq
Considering
\beeq
\label{0-1.2}
\begin{cases}
y''-\lambda^{2}K^{2}(r)y+G(r)y=0, r>\ep\la^{-1}\\
y(\ep \la^{-1})=y_0, y'(\ep \la^{-1})=y_{1}\in (0, \de_0^{-1}\la y_0 )
\end{cases}
\eneq
Then for any solution $y$
with $y, y'>0$,
we have the following uniform estimates, independent of $\la\in (0,\la_0]$,
\beeq
\label{0.1-eq-1.6}y
\simeq  y_0 e^{\lambda\int^{r}_{\ep/\la }K(\tau)d\tau}
\ ,\ r\ge \ep\la^{-1}\ .\eneq
Assume in addition
$1-\lambda^{-2}K^{-2}G\in (\de_0, \de_0^{-1})$,
then
the solution $y$ to
\eqref{0-1.2} satisfies
 $y, y'>0$ and
we have
\beeq
\label{0.1-eq-1.6'} y'\simeq y_1+y_0\la (e^{\lambda\int^{r}_{\ep/\la }K(\tau)d\tau}-1)\ .\eneq
\end{lem}
\begin{proof}[Proof of Lemma \ref{lem4}]
For $r\leq \de$, by the similar proof of Lemma \ref{lem1}, there is a solution $\phi_{\la_0}$ satisfying
$$\phi_{\la_0}\sim r^{\rho(v_0+\la_0d_0)}\ , \phi_{\la_0}\in H^1(B_\de)\cap W^{2, 1+}(B_\de),
\pa_r \phi_{\la_0}\ge 0
\ .$$
When $r\geq \de$, with $\phi_{\la_0}(x)=\phi_{\la_0}(r) $, we need only to consider the following  ordinary differential equation
\beeq
\label{618-1}
\phi''_{\la_0}+\frac{n-1}{r}\phi'_{\la_0}=(\la_0^2+\la_0 D+V)\phi_{\la_0}\ , \phi_{\la_0}(\de)=C_1>0,\ \pa_r\phi_{\la_0}(\de)=C_2\ge 0\ .
\eneq Then, as in the proof of Lemma \ref{lem1},  there is a unique solution $\phi_{\la_0}(r)\in C^2(0, \infty)$ and $\pa_r\phi_{\la_0}(r)\geq 0$ for all $r>0$.

For $r\geq R$, we shall consider \eqref{618-1} with $\phi_{\la_0}(R)>0$ and $\phi'_{\la_0}(R)\ge 0$.
 Let $y=r^{\frac{n-1}{2}}\phi_{\la_0}$, then the new function $y$ satisfies
\beeq
\begin{cases}
y''-(\la_0^2+\frac{d_\infty}{r}\la_0)y-(\frac{(n-1)(n-3)}{4r^2}+V+\la_0 D_\infty)y=0\\
y(R)=R^{\frac{n-1}{2}}\phi_{\la_0}(R), y'(R)\geq \frac{n-1}{2}R^{\frac{n-3}{2}}\phi_{\la_0}(R)> 0
\end{cases}
\eneq
Thus by Lemma \ref{thm-ode}, we have
$$y\simeq e^{\int^{r}_R (\sqrt{\la_0^2+\la_0\frac{d_\infty}{\tau}})d\tau}\simeq r^{\frac{d_\infty}{2}}e^{\la_0 r}, r\geq R\ ,$$
which yields
$$\phi_{\la_0}\simeq r^{-\frac{n-1-d_\infty}{2}}e^{\la_0 r}, r\geq R\ .$$
\end{proof}

\subsection{Proof of Lemma \ref{lem2}}Suppose $\la_0^2+V+\la_0 D$ is H\"older continuous in $\overline{B_\de}$ for some $\de>0$. Then we
consider the Dirichlet problem
\beeq
\label{617-3}
\begin{cases}
\Delta\phi_{\la_0}=(\la_0^2+\la_0 D+V)\phi_{\la_0}, x\in B_{\de}\\
\phi_{\la_0}|_{\pa B_{\de}}=1.
\end{cases}
\eneq
By Theorem 6.14 in Gilbarg-Trudinger \cite{GT}, there exists a unique $C^2(\overline{B_\de})$ solution. Hence we could apply maximum principle to get $0< \phi_{\la_0}\leq 1$ in $B_\de$ and $\pa_{r}\phi_{\la_0}\ge 0$ for all $0<r\leq \de$. When $r\geq \de$,  with $\phi_{\la_0}(x)=\phi_{\la_0}(r)$, \eqref{617-3} becomes an ordinary differential equation
\beeq
\label{617}
\begin{cases}
\phi''_{\la_0}+\frac{n-1}{r}\phi'_{\la_0}-(\la_0^2+\la_0 D+V)\phi_{\la_0}=0\\
\phi_{\la_0}(\de)=1, \phi'_{\la_0}(\de)=C_2\ge 0.
\end{cases}
\eneq
Then  $\phi_{\la_0}\in C^2(\R^n)$ and we could apply strong maximum principle again to get $\pa_r\phi_{\la_0}\ge 0$ for all $r>0$.

Furthermore, for $r\geq R$, we consider \eqref{617} from $r=R$, then we have $\phi_{\la_0}(R)>0, \phi'_{\la_0}(R)\geq 0$.
By the same argument as in the proof of Lemma \ref{lem4}, we get
$$\phi_{\la_0}\simeq r^{-\frac{n-1-d_\infty}{2}}e^{\la_0 r}, r\geq R\ ,$$
which completes the proof.

\subsection{Proof of Lemma \ref{elp}}

%\begin{proof}

We first show that \eqref{509} admits a $C^2(\R^n)$ solution. For $r\leq \la^{-1}$, we consider the Dirichlet problem within $B_{ \la^{-1}}$
\beeq
\label{Diri}
\begin{cases}
\Delta\phi_\la-\la D(x)\phi_\la=\lambda^{2}\phi_\la, x\in B_{ \la^{-1}}\\
\phi_\la|_{\pa B_{\la^{-1}}}=1\ .
\end{cases}
\eneq
By Theorem 6.14 in Gilbarg-Trudinger \cite{GT}, there exists a unique (and hence radial) $C^2(B_{ \la^{-1}})$ solution. For $r>0$, the equation is reduced \eqref{509} to a second order ordinary differential equation
$$\phi''_\la+\frac{n-1}{r}\phi'_{\la}-\la D(r)\phi_{\la}=\la^2\phi_\la\ ,$$
which ensures that $\phi_\la\in C^{2}(\R^n)$. To obtain \eqref{510}, we need to divide $\R^n$ into two parts: $B_{1/\la}$ and $\R^n\backslash B_{1/\la}$.\\
% Without loss of generality, we assume $\phi_\la(1/\la)=1$.\\

{(\bf \uppercase\expandafter{\romannumeral1}) Inside the ball $B_{1/\lambda}$}\\

Considering the Dirichlet problem \eqref{Diri} within $B_{1/\lambda}$,
it is easy to see that $0< \phi_\la\leq 1$ in $B_{1/\lambda}$. In fact, if there exists $x_{0}\in B_{1/\lambda}$ such that $\phi_\la(x_{0})\leq 0$, then by strong maximum principle, we get $\phi_\la$ is constant within $B_{1/\lambda}$, which is a contradiction. By Hopf's lemma, we have $\pa_{r}\phi_\la>0$ for $0<r\leq 1/\la$.

To get the uniform lower bound of $\phi_\la$, we define rescaled function $f_{\lambda}(x)=\phi_\la(x/\la)$, which satisfies
\begin{align*}
\begin{cases}
\Delta f_{\lambda}-\frac{1}{\la}D(\frac{r}{\la})f_\la=f_{\lambda}, x\in B_{1}\\
f_{\lambda}|_{\pa B_{1}}=1\ .
\end{cases}
\end{align*}
Since $f_{\lambda}$ is radial increasing, we know that
$$\liminf_{0<\lambda\leq 1}\inf_{x\in B_{1}}f_{\lambda}=\liminf_{0<\lambda\leq 1}f_{\lambda}(0):=C\ge 0\ .$$
To complete the proof, we need only to prove
$C>0$. By definition, there exists a sequence $\lambda_{j}\to 0$ such that $f_{\lambda_{j}}(0)\to C$ as $j\to \infty$.\par
{(\expandafter{\romannumeral1}) Derivative estimates of $f_\la$}.

As $f_\la$ is radial, we have
$$\Delta f_\la=r^{1-n}\pa_{r}( r^{n-1}\pa_{r}f_\la)\ ,$$
and so
\beeq
\label{s01}
\pa_{r}( r^{n-1}\pa_{r}f_\la)=\Big(\frac{1}{\la}D(\frac{r}{\la})+1\Big)r^{n-1}f_\la.
\eneq
Recall that $f_\la\in(0, 1]$ and $D(r)\leq \frac{\mu}{(1+|x|)^{\be}}$, by integrating it from $0$ to $r$, we get
\begin{align*}
r^{n-1}\pa_{r}f_\la=
&\int_0^r\Big(\frac{1}{\la}D(\frac{\tau}{\la})+1\Big)\tau^{n-1}f_\la d\tau\\
\leq &\int_0^r\Big(\frac{\mu \la^{\be-1}}{(\la+\tau)^{\be}}+1\Big)\tau^{n-1}f_\la d\tau\\
\leq&\int_0^r\Big(\frac{\mu}{\tau}+1\Big) \tau^{n-1} d\tau\\
\leq&\frac{\mu}{n-1}r^{n-1}+\frac{r^n}{n}\ ,
\end{align*}
that is,
\beeq
\label{yiyou}
\pa_{r}f_{\lambda}
\le\frac{\mu}{n-1}+\frac{r}{n}\ .
\eneq

{(\expandafter{\romannumeral2}) Convergence of $f_\la$}.

Since $\|f_\la\|_{H^{1}(B_{1})}$ are uniformly bounded, there exists a subsequence of $\lambda_{j}$ (for simplicity we still denote the subsequence as $\lambda_{j}$) such that
$f_{\lambda_{j}}$ converges weakly to some
$f$ in $H^{1}(B_1)$ as $j$ goes to infinity. Moreover, by the Arzela-Ascoli theorem, $f_{\lambda_{j}}$ converges uniformly to
$f$ in $C(\bar B_{1})$ as $j$ goes to infinity, thus $f=1$ on $\pa B_1$ and $f(0)=C$.

In view of the equations satisfied by $f_{\la_j}$, we see that, for any $\phi\in C^\infty_c(B_1)$, we have
\beeq
\label{eq-1027}
\int_{B_1} \nabla f_{\la_j} \nabla\phi +f_{\la_j} \phi+\frac{1}{\la_j}D\big(\frac{r}{\la_j}\big)f_{\la_j}\phi   d x=0\ .
\eneq
Let $g_{\la_j}=\frac{1}{\la_j}D(\frac{r}{\la_j}) f_{\la_j}\phi$, then it is easy to see
$$|g_{\la_j}|\leq |\phi|\frac{\mu\la_{j}^{\be-1}}{(\la_j+|x|)^{\be}}\to 0, \ \forall x\in B_{1}\backslash\{0\} \ , $$
as $\la_j \to 0$.
Notice that
$$|g_{\la_j}|\le\ \frac{\mu |\phi|}{\la_j+|x|}\ \le \ \frac{\mu \|\phi\|_{L^\infty}}{|x|}\in L^{1}(B_{1})\ ,$$
by dominated convergence theorem, we have
$$\lim_{\la_j\to0}\int_{B_1}g_{\la_j}dx=0\ .$$
Thus let $\la_j \to 0$ in \eqref{eq-1027}, we get
\beeq\label{eq-1027.1}\int_{B_1} (\nabla f \cdot \nabla \phi +f \phi) d x=0\ ,\eneq for any $\phi\in C^\infty_c(B_1)$. This tells us that $f\in H^1(B_1)$ is a weak solution to the Poisson equation
\begin{align*}
\begin{cases}
\Delta f=f, x\in B_{1},\\
f|_{\pa B_{1}}=1\ ,
f(0)=C\ .
\end{cases}
\end{align*}
By regularity and
strong maximum principle, we know that $f\in C^\infty(B_1)$ and
$f(0)=C>0$, which completes the proof of the claim $C>0$.\\

{\bf (\uppercase\expandafter{\romannumeral2}) Outside the ball $B_{1/\la}$}\\

Inspired by Lemma \ref{thm-ode}, we try to reduce \eqref{509} to a second order ordinary differential equation by finding a radially symmetric solution when $r\geq 1/\lambda$.
Before proceeding, we need to estimate the derivative of $\phi_\la$. Recall that $f_\la(r)=\phi_\la(\frac{r}{\la})$, then by \eqref{yiyou}, we have
$$(\pa_r\phi_\la)(\la^{-1})=\la\pa_r f_\la(1)\leq \left(\frac{\mu}{n-1}+\frac{1}{n}\right)\la=C_1 \la\ .$$
\par
We consider the second order ordinary differential equation
\beeq
\begin{cases}
\phi''_\la+\frac{n-1}{r}\phi'_{\la}-\la D(r)\phi_{\la}=\la^2\phi_\la\\
\phi_{\la}(\frac{1}{\la})=1, \phi'_{\la}(\frac{1}{\la})\in (0, C_1\la]\ .
\end{cases}
\eneq
Let $\phi_{\la}(r)=r^{-\frac{n-1}{2}}y(r)$, then $y$ satisfies
\beeq
y''-\la^{2}y-\Big(\frac{(n-1)(n-3)}{4r^{2}}+\la D(r)\Big)y=0
\eneq
with initial data
$$y(\la^{-1})=\la^{-\frac{n-1}{2}}, y'(\la^{-1})=\frac{n-1}{2}\la^{-\frac{n-3}{2}}+\la^{-\frac{n-1}{2}}\phi_\la'(\la^{-1})\in  (0, C_2\la y(\la^{-1}) ]\ ,$$
where $C_2=\frac{n-1}{2}+C_1$.
Thus by Lemma \ref{thm-ode} with $K=1, \ep=1, \la_0=1$, we have
$$y\simeq \la^{-\frac{n-1}{2}}e^{\la r}, r\la\geq 1\ ,$$
which yields
$$\phi_{\la}(x)\simeq (\la|x|)^{-\frac{n-1}{2}}e^{\la |x|}, \la|x|\geq 1\ .$$

Combining (I), (II),
 we conclude that there exist  uniform $c_1\in(0, 1)$ and  solutions $\phi_\la$ of \eqref{509} with $\la\in (0,1]$ satisfying the uniform estimates \eqref{510}.

\section{Proof of Theorem \ref{general} }\label{sec:AE3}

In this section, we prove Theorem \ref{general}.

\subsection{Test function method}
Equipped with the test functions, we could construct
two kinds of radial solutions to the linear dual problem $$(\pt^2-\Delta-D\pt+V)\Phi=0\ ,$$
that is,
 $$\Phi_0(t,x)=\phi_0(x),\
\Phi_{\la}(t,x)=e^{-\la t}\phi_\la (x)\ .$$

In addition to these solutions, we will also introduce a smooth cut-off function.
Let $\eta(t)\in C^{\infty}([0, \infty))$ such that
$$\eta(t)=1, t\leq \frac{1}{2}\ , \ \ \eta(t)=0, t\geq 1\ .$$
Then for $T\in (2, T_{\vep})$, we set $\eta_T(t)=\eta(t/T)$.

Let $\Psi_\la=\eta_T^{2p'}(t)\Phi_\la(t,x)$, where $\Phi_\la\in C^2([0, T]\times (\R^n\backslash\{0\}))\cap  C^0_t H_{loc}^1\cap C^1_t L_{loc}^2)([0, T]\times \R^n)$.
Then by the definition of energy solution \eqref{weaksol},
 we have
\begin{eqnarray*}
&&%\vep\int_{\R^n}(g(x)+D(x)f(x))dx+
\int_0^T\int_{\R^n}|u|^p\Psi_{\la} dxdt
-\left.\int_{\R^n}(u_t(t,x)+D(x)u(t,x))\Psi_{\la}(t, x)dx\right|_{t=0}^T\\
 \\
 & = &  -\int_0^T\int_{\R^n}u_t(t, x)\pt\Psi_{\la}(t, x)dxdt+\int_0^T\int_{\R^n}\nabla u(t, x)\cdot\nabla\Psi_{\la}(t, x)dxdt\\
&&-\int_0^T\int_{\R^n}D(x)u(t, x)\pt\Psi_{\la}(t, x)dxdt+\int_0^T\int_{\R^n}V(x)u(t, x)\Psi_{\la}(t, x)dxdt\\
&=&
-\int_{\R^n}u(t, x)\pt\Psi_\la(t, x)dx|_{t=0}^T+
\int_0^T\int_{B_{t+R}}u(\pt^2
-\Delta -D\pt+V)\Psi_\la
dxdt\ ,
\end{eqnarray*}
where all of the integration by parts and integrals could be justified by the properties of $\Phi_\la$ and the support assumption $\supp u(t,\cdot)\subset B_{t+R}$.

Noticing that $$\Psi_\la(T)=\pt\Psi_\la(T)=0,
\pt \Psi_\la=-\la \Psi_\la+\pt (\eta_T^{2p'})\Phi_\la\ ,$$
\[
\begin{aligned}
&\partial_t\eta_T^{2p'}=\frac{2p'}{T}\eta_T^{2p'-1}\eta'(\f{t}{T})=\mathcal{O}(\frac{\eta_T^{2p'-1}}{T}\chi_{[\frac T2,T]}(t)),\\
&\partial_t^2\eta_T^{2p'}=\frac{2p'(2p'-1)}{T^2}\eta_T^{2p'-2}|\eta'|^2+\frac{2p'}{T^2}\eta_T^{2p'-1}\eta''=\mathcal{O}(\frac{{\eta_T^{2(p'-1)}}}{T^2}\chi_{[\frac T2,T]}(t))\ .
\end{aligned}
\]
The integral identity could be reorganized into the following form:
\begin{eqnarray}
&&%\vep\int_{\R^n}(g(x)+D(x)f(x))dx+
\int_0^T\int_{\R^n}|u|^p\Psi_{\la} dxdt
+\vep\int_{\R^n}(g(x)+(\la+D(x))f(x))\Phi_{\la}( x)dx\nonumber\\
&=&
%-\int_{\R^n}u(t, x)\pt\Psi_\la(t, x)dx|_{t=0}^T+
\int_0^T\int_{B_{t+R}}u(\pt^2
-\Delta -D\pt+V)\Psi_\la
dxdt\nonumber\\
&=&
\int_0^T\int_{B_{t+R}}u(\pt^2 (\eta_T^{2p'})+2\pt (\eta_T^{2p'})\pt
-D\pt (\eta_T^{2p'}))\Phi_\la
dxdt\nonumber
\\
&=&
\int_0^T\int_{B_{t+R}}u(\pt^2 (\eta_T^{2p'})-2\la \pt (\eta_T^{2p'})
-D\pt (\eta_T^{2p'}))\Phi_\la
dxdt\nonumber
\\
&\le &C
\int_{T/2}^T\int_{B_{t+R}}|u| \eta_T^{2(p'-1)}(T^{-2}+
(2\la+|D|) T^{-1})\Phi_\la
dxdt
\ .\label{eq-key}
\end{eqnarray}

Basically, the test function method is to construct specific test function, so that we could try to use the integral inequality to control the right hand side by the left hand side, which gives the lifespan estimates.

Before proceeding, let us present the following  technical Lemma.
\begin{lem}\label{keyinq}
Let $\beta>0$, $\alpha,\ \gamma\in \R$ and $R>0$,
there exists a constant $C$, independent of $t>2$, so that
\begin{equation}
\label{keyinq0}
\begin{aligned}
\int_0^{t+R}(1+r)^{\alpha}\ln^{\gamma}(1+r)e^{-\beta(t-r)}dr\le C(t+R)^{\alpha}\ln^{\gamma}(t+R).
\end{aligned}
\end{equation}
\end{lem}
\begin{prf} We split the proof into two parts. First it is easy to see
\begin{eqnarray*}
\int_{\frac{t+R}{2}}^{t+R}(1+r)^{\alpha}\ln^{\gamma}(1+r)e^{-\beta(t-r)}dr &\le& Ce^{-\beta t}(t+R)^{\alpha}\ln^{\gamma}(t+R)\int_{\frac{t+R}{2}}^{t+R}e^{\beta r}dr\nonumber\\
&\le& C(t+R)^{\alpha}\ln^{\gamma}(t+R).%\label{keyinq1}
\end{eqnarray*}
For the remaining case, we have
%\begin{equation}\label{keyinq2}
%\begin{aligned}
\begin{eqnarray*}
\int_0^{\frac{t+R}{2}}(1+r)^{\alpha}\ln^{\ga}(1+r)e^{-\beta(t-r)}dr&
\le& Ce^{-\frac{\beta t}{2}}\int_0^{\frac{t+R}{2}}(1+r)^{|\alpha|+1} dr\\
&\le& Ce^{-\frac{\beta t}{2}} (t+R)^{|\alpha|+2}\\
& \le& C(t+R)^{\alpha}\ln^\ga(t+R)\ ,\end{eqnarray*}
which completes the proof.
\end{prf}

\subsection{First choice of the test function $\Phi_0(t,x)$ with $\la=0$}\label{sec:1st}
Let $\la=0$, we have $\phi_0$ ensured by
 Lemma \ref{lem1}.
 % and Lemma \ref{lem8}.
  With help of $\Phi_0=\phi_0$, the inequality \eqref{eq-key} reads as follows
\begin{eqnarray*}
&&C_1(f, g)\vep+\int_0^T\int_{\R^n}|u|^p\eta_T^{2p'}\phi_0dxdt\\
&\le&
C \int_0^T\int_{\R^n}|u| \eta_T^{2(p'-1)}(T^{-2}+
|D| T^{-1})\phi_0
dxdt\\
&\le&
\frac 12\int_0^T\int_{\R^n}|u|^p\eta_T^{2p'}\phi_0 dx dt +C
\int_{T/2}^T\int_{B_{t+R}}
 (T^{-2}+
|D| T^{-1})^{p'}\phi_0
dxdt \ ,
\end{eqnarray*}
where $C_1(f, g)=\int_{\R^n}(g(x)+D(x)f(x))dx$
and
we have used the H\"older and Young's inequality in last inequality.
In conclusion,
\beeq
\label{step1}
C_1(f, g)\vep+\int_0^T\int_{\R^n}|u|^p\eta_T^{2p'}\phi_0dxdt\le C
\int_{T/2}^T\int_{B_{t+R}}
 (T^{-2}+
|D| T^{-1})^{p'}\phi_0
dxdt=F_0 \ ,
\eneq

Concerning the right hand side of \eqref{step1},
 by our assumption $D=\CO(r^{\theta-2})$ locally,
 $D=\CO(r^{-1})$ near spatial infinity, and Lemma \ref{lem1},
 we see that $F_0$ could be controlled by
\begin{eqnarray*}
&& T^{-p'} \int_{T/2}^T\int_0^{1}(T^{-1}+r^{\theta-2})^{p'} r^{\rho(v_0)+n-1}drdt\\
&&+T^{-p'}\int_{T/2}^T\int_1^{t+R}(T^{-1}+r^{-1})^{p'}r^{\rho(v_\infty)+n-1}drdt \\
&\les&T^{-2p'}\left(T+T^{\rho(v_\infty)+n+1}\right)+T^{1-p'}
+T^{1-p'}\int_1^{T+R}r^{-p'+\rho(v_\infty)+n-1}dr
\\
&\les&
\left\{\begin{array}{ll}
\max\{T^{1-2p'},T^{-2p'+\rho(v_\infty)+n+1},T^{1-p'}\}& p'\neq\rho(v_\infty)+n\\
T^{1-p'}\ln T & p'=\rho(v_\infty)+n
\end{array}
\right.
\\
&=&
\left\{\begin{array}{ll}
T^{-2p'+\rho(v_\infty)+n+1} & p'<\rho(v_\infty)+n\\
T^{1-p'}\ln T & p'=\rho(v_\infty)+n\\
T^{1-p'} & p'>\rho(v_\infty)+n\ ,
\end{array}
\right.
\end{eqnarray*}
where, %we have used the assumption
%$\rho(v_\infty)+n>0
 to ensure the integrability of first term of second bracket, we need to assume
 \beeq\label{eq-p2}
 (\theta-2)p'+\rho(v_0)+n>0\Leftrightarrow
 p>p_2:=
 \f{n+\rho(v_0)}{n+\rho(v_0)+\theta-2}\ .
% \left\{\begin{array}{ll}\f{n+\rho(v_0)}{n+\rho(v_0)+\theta-2},& n+\rho(v_0)+\theta>2, \\ 1 & n+\rho(v_0)+\theta=2 .\end{array}\right.
\eneq
Also, we observe that
\beeq\label{eq-p3} p'<\rho(v_\infty)+n\Leftrightarrow
p>p_3:=\f{n+\rho(v_\infty)}{n+\rho(v_\infty)-1}\ .
\eneq
In conclusion, provided that $p>p_2$, we have
\begin{equation}\label{F0}
F_0\les \left\{\begin{array}{ll}
T^{-2p'+\rho(v_\infty)+n+1} & p>\max(p_2, p_3)\\
T^{1-p'}\ln T  & p=p_3>p_2\\
T^{1-p'} &
p_2<p<p_3, \mathrm{if}\ p_2<p_3\ .
\end{array}
\right.
\end{equation}

Recalling \eqref{step1}, if we have data such that $C_1(f, g)>0$, which is always true for $f=0$ and
nontrivial, nonnegative and compactly supported
$g$, we have
$\vep\les F_0$. Then
we obtain the blow up results for $p_2<p<p_G$, whenever $(p_2, p_G(n+\rho(v_\infty)))\neq \emptyset$, and
 $$
-2p'+\rho(v_\infty)+n+1<0\Leftrightarrow p<p_G(n+\rho(v_\infty))
%=1+\frac{2}{n+\rho(v_\infty)-1}
\ .
$$ At the same time, we could extract the first three lifespan estimates in \eqref{lifespan1}.

\subsection{Second choice of the test function $\Phi_1(t,x)$ with $\la=1$}
Let $\la=1$, we have $\phi_1$ ensured by
 Lemma \ref{lem4} with $\la_0=1$.
 % and Lemma \ref{lem8}.
  With help of $\Phi_1=e^{-t}\phi_1$,  $\Psi_1=\eta_T^{2p'}(t)\Phi_1$,
  as for  \eqref{step1},   the inequality \eqref{eq-key} gives us
\begin{eqnarray}
&&\int_0^T\int_{\R^n}|u|^p\Psi_1 dxdt + C_2(f,g) \vep
%\int_{\R^n}(g+f+D(x)f)\phi_1dx
\nonumber\\
&\le &C
\int_{T/2}^T\int_{B_{t+R}}|u| \eta_T^{2(p'-1)} %T^{-2}+
\frac {2+|D|} T\Phi_1
dxdt\label{middle}\\
&\le&
\frac 12 \int_{T/2}^T\int_{\R^n}|u|^p \eta_T^{2p'}
\Phi_1
dxdt+C\int_{T/2}^T\int_{B_{t+R}}
\left(
\frac {2+|D|} T\right)^{p'}\Phi_1
dxdt
% C\int_{T/2}^T\int_{\R^n} (T^{-2}+(2+|D|) T^{-1})^{p'}\phi_1dxdt=F_1
\ ,\nonumber
\end{eqnarray}
and so
\beeq
\label{small_p}
%\label{middle1}%\int_0^T\int_{\R^n}|u|^p\Psi_1 dxdt +
C_2(f,g) \vep \ %\int_{\R^n}(g+(1+D)f)\phi_1dx
\les
\int_{T/2}^T\int_{B_{t+R}}
\left(
\frac {2+|D|} T\right)^{p'}\Phi_1
dxdt\ ,\eneq
where $C_2(f,g):=\int_{\R^n}\phi_1(Df+f+g)dx$.

As in the previous section \ref{sec:1st},
we could extract certain lifespan estimate from \eqref{small_p}.
Actually,
by \eqref{small_p}, Lemma \ref{keyinq} and  Lemma \ref{lem4} with $\la_0=1$, we have
\begin{equation}\label{lastcase}
\begin{aligned}
\vep \les& T^{-p'}\int_{\f{T}2}^T\int_{|x|\leq t+R}(2+|D(x)|^{p'})\Phi_1dxdt\\
\les& T^{-p'}\Bigg(\int_{\frac T2}^T\int_{r\le 1} (1+r^{(\theta-2)p'})e^{-t}r^{\rho(v_0+ d_0)+n-1}drdt\\
&+\int_{\frac T2}^T\int_{1\le r\le t+R}(1+r^{-p'})e^{-t}e^{r}r^{\f{1-n+d_\infty}{2}+n-1}drdt\Bigg)\\
\les&T^{-p'}\left(e^{-\frac T2}+T^{\f{n+1+d_\infty}{2}}\right)
\les T^{-p'+\f{n+1+d_\infty}{2}},\\
\end{aligned}
\end{equation}
provided that $C_2(f,g)>0$, and
\[
\rho(v_0+d_0)+n+(\theta-2)p'>0,  i.e.,
p>p_4:=\f{n+\rho(v_0+d_0)}{n+\rho(v_0+ d_0)+\theta-2}\ .
\]

Based on \eqref{lastcase},
when
$$-p'+\f{n+1+d_\infty}{2}<0\Leftrightarrow
p<p_G(n+d_\infty)\ ,$$
 we obtain
the last lifespan estimate in \eqref{lifespan1}
for any $p\in (p_4, p_G(n+d_\infty))\neq \emptyset$,
whenever
$f=0$ and nontrivial $g\ge 0$.

\subsection{Combination}
It turns out that we have not exploited the full strength of $\Psi_0$ and $\Psi_1$. Actually, a combination of
\eqref{step1} and
\eqref{middle} could give us more information on the lifespan estimates.

To connect \eqref{middle} with that appeared in \eqref{step1},
we try to control
the middle term in
 \eqref{middle}  by the left of \eqref{step1}, that is,
\begin{equation*}
\begin{aligned}
\vep T  \les& \int_{T/2}^T\int_{B_{t+R}}|u| \eta_T^{2(p'-1)}
(2+|D|)\Phi_1
dxdt
\\
\les& \left(\int_0^T\int_{\R^n}\eta_T^{2p'}\phi_0|u|^p dxdt\right)^{\frac1p}\left(\int_\f{T}2^T\int_{|x|\leq t+R}(2+|D(x)|)^{p'}\phi_0^{-\f{p'}{p}}\Phi_1^{p'}dxdt\right)^{\frac{1}{p'}}.
\end{aligned}
\end{equation*}
Thus, combining it with \eqref{step1}, we derive that
\beeq\label{large_p}
\vep^p T^p  \les  F_0\left(\int_\f{T}2^T\int_{|x|\leq t+R}(2+|D(x)|)^{p'}\phi_0^{-\f{p'}{p}}\Phi_1^{p'}dxdt\right)^{p-1}:=F_0 F_1^{p-1}\ . \eneq

Concerning $F_1$,  by Lemma \ref{lem1} and \ref{lem4}, we have
\begin{eqnarray*}
F_1&=&\int_{\frac T2}^T\int_{|x|\leq t+R}(2+|D|^{p'})\phi_0^{-\f{p'}{p}}\Psi_0^{p'}dxdt \\
&\les&\int_{\frac T2}^Te^{-t p'}\int_0^1(2+r^{(\theta-2)p'})r^{-\f{p'}{p}\rho(v_0)+p'\rho(v_0+d_0)+n-1}drdt\\
&&+\int_{\frac T2}^T\int_1^{t+R}(2+|D|)^{p'}r^{-\f{p'}{p}\rho(v_\infty)-p'\f{ n-1-d_\infty }{2}+n-1}e^{(r-t)p'}drdt.
\end{eqnarray*}

To ensure the integrability near $r=0$, we need to require that
\[(\theta-2)p'-\f{p'}{p}\rho(v_0)+p'\rho(v_0+d_0)+n-1>-1,\]
 that is
\begin{equation}\label{p-largerange}
p>p_1:=\f{n+\rho(v_0)}{n+\rho(v_0+d_0)+\theta-2}\ .
\end{equation}
While for the second integral, we utilize Lemma \ref{keyinq}  with  $\be=1$, as well as $|D|\les r^{-1}$ for $r>1$, to conclude
\begin{eqnarray*}
&&\int_1^{t+R}(2+|D|)^{p'}r^{-\f{p'}{p}\rho(v_\infty)-p'\f{(n-1-d_\infty)}{2}+n-1}e^{(r-t)\la_0 p'}dr\\
&\les&\int_1^{t+R}r^{-\f{p'}{p}\rho(v_\infty)-p'\f{n-1-d_\infty}{2}+n-1}e^{(r-t)\la_0 p'}dr\\
&\les&(t+R)^{-\f{p'}{p}\rho(v_\infty)-p'\f{n-1-d_\infty}{2}+n-1}.
\end{eqnarray*}
Thus we have
\beeq\label{F1}F_1\les e^{-Tp'/2}+T^{-\f{p'}{p}\rho(v_\infty)-p'\f{n-1-d_\infty}{2}+n}\les T^{-\frac{\rho(v_\infty)}{p-1}-p\f{n-1-d_\infty}{2(p-1)}+n}\ ,\eneq
if $p>p_1$.

In view of \eqref{F0}, \eqref{F1} and \eqref{large_p}, we arrive at, for $p>\max(p_1, p_2)$,
$$\vep^p  \les T^{-p}  F_0 F_1^{p-1}
%\les
%T^{ \rho(v_\infty)-p\f{n+1-d_\infty}{2}+n(p-1)}F_0%same as the next line
%T^{- \rho(v_\infty)+p\f{n+d_\infty-1}{2}-n}F_0
\les \left\{\begin{array}{ll}
T^{ \f{n+d_\infty-1}{2}p-2p' +1} & p> p_3\\
T^{\f{n+d_\infty-1}{2}p-2p' +1}\ln T  & p=p_3\\
T^{-\rho(v_\infty)+p\f{n+d_\infty-1}{2}-n+1-p'} &
p<p_3\ ,
\end{array}
\right.
$$ by which we are able to obtain the final set of lifespan estimates.

When $p>\max\{p_1,p_2,p_3\}$, we could obtain an upper bound of the lifespan, if
$$\f{n+d_\infty-1}{2}p-2p' +1<0
\Leftrightarrow
\f{n+d_\infty-1}{2}p(p-1)<p+1,
$$ that is,
$p<p_{S}(n+d_\infty)$. This gives us the
sixth lifespan estimate in \eqref{lifespan1}:
$$T_\vep\  \les  \vep^{-\f{2p(p-1)}{\gamma(p,n+d_\infty)}}.$$

In the critical case, $p=p_3\in (\max\{p_1,p_2\}, p_{S}(n+d_\infty))$, we obtain the estimate with log loss:
$$T_\vep\les \vep^{-\f{2p(p-1)}{\gamma(p,n+{d_\infty})}}(\ln \vep^{-1})^{\frac{2(p-1)}{\gamma(p,n+{d_\infty})}}\ . $$

For the remaining case of $\max\{p_1,p_2\}<p<p_3$,
we could obtain an upper bound of the lifespan, provided that
${-\rho(v_\infty)+p\f{n+d_\infty-1}{2}-n+1-p'}<0$, i.e.,
$$
\f{n+d_\infty-1}{2}p(p-1)
+(n+\rho(v_\infty)-1)(p_3-p)
<p+1\Leftrightarrow
\ga_2>0\Leftrightarrow p<p_5\ ,
%2(n+\rho(v_\infty)-1)(p_3-p)<\ga(p, n+d_\infty)
$$ where $\ga_2$ is given in \eqref{eq-gamma_2}.
Thus we have blow up result and lifespan estimate,
the
fourth lifespan estimate in \eqref{lifespan1},
 if $\max\{p_1,p_2\}<p<\min(p_3, p_5)$. %where $p_5$ is given in

Finally, we remark that
%\marginpar{\color{red}to put in main text}{\color{blue}
%In addition,
the last upper bound of the lifespan in \eqref{lifespan1} was obtained
 for $p\in(p_4, p_G(n+d_\infty))\neq \emptyset$. However, after comparison with the corresponding estimates in the first and fourth case, we see that it gives better upper bound only for
$p\in (p_4, p_0]\cap (p_4, p_G(n+d_\infty))$,
 if we have
$p_4<\min (p_0, p_G(n+d_\infty))$. This is the reason we state it
for this restricted range in \eqref{lifespan1}.

\section{Proof of Theorem  \ref{regular} }\label{sec:regular}

In this section, we prove Theorem \ref{regular}.

%\subsection{Proof of Theorem \ref{regular}}

 According to Lemma \ref{lem8} and \ref{lem2} we have,
\[
\phi_0\simeq
\begin{cases}
\ln(r+2)\ , n=2\\
1\ , \ n\geq 3\ ,
\end{cases}
\
 \mbox{and}
\ \
\phi_1\simeq
\langle r\rangle^{-\frac{n-1-d_\infty}{2}}e^{ r}\ .
\]
Thus, when $n\ge 3$, the situation considered in Theorem \ref{regular} could be viewed as a particular case of $v_0=d_0=v_\infty=0$  and $\theta=2$, in Theorem \ref{general}. Then $p_1=p_2=1$, $p_3=\f{n}{n-1}$, and so Theorem \ref{regular} for $n\ge 3$ could be obtained as the Corollary of Theorem \ref{general}.
In the following, it suffices for us to  present the proof for $n=2$, for which we still have
\eqref{step1},
\eqref{middle},
 \eqref{small_p} and
 \eqref{large_p}, with only replacements from $\phi_0\sim 1$ to $\phi_0\sim
 \ln(r+2)$.

Concerning the right hand side of \eqref{step1},
 we have
 \begin{eqnarray}\label{F0-regu}
\vep\ \les F_0&\les &T^{-p'}\int_{T/2}^T\int_0^{t+R}(T^{-1}+\<r
\>^{-1})^{p'}r\ln (r+2)  drdt \nonumber\\
&\les&T^{-2p'}T^{3}\ln T +T^{-p'}\int_{T/2}^T\int_0^{t+R}\<r
\>^{1-p'} \ln (r+2)  drdt \nonumber
\\
&\les&
\left\{\begin{array}{ll}
T^{1-p'} & p<2\\
T^{1-p'}(\ln T)^2 & p=2\\
T^{3-2p'}\ln T & p>2\ .
\end{array}
\right.
\end{eqnarray}
Based on this inequality, we could extract the first three lifespan estimates in \eqref{lifespan1a} for $p\in (1,3)$ and $n=2$, which have certain log loss for the case $p\ge n/(n-1)=2$.

Concerning $F_1$ in
\eqref{large_p},   by  Lemma \ref{keyinq}, we have
\begin{eqnarray*}
F_1&=&\int_{\frac T2}^T\int_{|x|\leq t+R}(2+|D|^{p'})\phi_0^{-\f{p'}{p}}\Psi_0^{p'}dxdt \\
&\les&\int_{\frac T2}^T\int_0^{t+R}(1+r)^{-p'\f{ 1-d_\infty }{2}+1}
(\ln (r+2))^{-\f{p'}{p}}
 e^{(r-t)p'}drdt\\
 &\les& T^{2-p'\f{ 1-d_\infty }{2}}
(\ln T)^{-\f{p'}{p}}\ .
\end{eqnarray*}
Plugging these estimates in \eqref{large_p}, we derive that
\beeq\label{F0F1}
\vep^p T^p  \les  F_0 F_1^{p-1}\les
T^{2(p-1)-p\f{ 1-d_\infty }{2}}\times
\left\{\begin{array}{ll}
T^{1-p'}(\ln T)^{-1} & p<2\\
T^{1-p'} \ln T & p=2\\
T^{3-2p'} & p>2\ .
\end{array}
\right.
\ . \eneq
Similarly, based on this inequality, we could extract the second set (fourth to sixth) of lifespan estimates in \eqref{lifespan1a} for $p\in (1,p_S(n+d_\infty))$ and $n=2$, which have certain log adjustment for the case $p< n/(n-1)=2$.

\section{Proof of Theorem \ref{thmStrauss}}\label{sec:thmStrauss}

\subsection{Subcritical case}
In this subsection, we present the proof of Theorem \ref{thmStrauss} for $p<p_S(n)$, under the assumption that
$V=0$ and $D(x)\in C(\R^n)\cap C^{\de}(B_\de)$, which is of short range, in the sense that 
$D\in L^n\cap L^\infty(\R^n)$.
%$D(r)\in L^1([0,\infty), dr)$.
The proof follows the same lines as that of Theorem \ref{general} or Theorem  \ref{regular}. 
%Actually, when $n\ge 3$ and $p\neq n/(n-1)$, it %is covered by  can be viewed as a corollary of Theorem  \ref{regular}.

At first, with $\phi_0=1$, by \eqref{step1}, we have
\beeq
\label{step1'}
C_1(f, g)\vep+\int_0^T\int_{\R^n}|u|^p\eta_T^{2p'} dxdt\le C
\int_{T/2}^T\int_{B_{t+R}}
 (T^{-2}+
|D| T^{-1})^{p'}
dxdt=F_0 \ .
\eneq
When $p'\ge n$, i.e., $p\le \frac{n}{n-1}$, we know that
$D\in L^n\cap L^\infty(\R^n)\subset L^{p'}$,
and so
$$F_0= C
\int_{T/2}^T\int_{B_{t+R}}
 (T^{-2}+
|D| T^{-1})^{p'}
dxdt \les 
T^{n+1-2p'}+T^{1-p'}\les T^{1-p'}.
$$
Otherwise, for
$p'< n$, i.e., $p> \frac{n}{n-1}$, by using H\"older's inequality, we obtain
$$\int_{B_{T+R}}
|D|^{p'}
dx\les
\||D|^{p'}\|_{L^{n/p'}}\|1\|_{L^{n/(n-p')}(B_{T+R})}\les T^{n-p'}\ ,$$
and thus
$$F_0\les 
T^{n+1-2p'}+T^{1-p'}
\int_{B_{T+R}}
|D|^{p'}
dx
\les T^{n+1-2p'}\ .
$$
In conclusion, we have
\begin{equation}\label{T12step2}
\int_0^T\int_{\R^n}|u|^p\eta_T^{2p'}dxdt\lesssim\left\{
\begin{array}{ll}
T^{1-p'}\ &\mbox{if}\ 1<p\le \frac{n}{n-1},\\
T^{n+1-2p'}\ &\mbox{if}\ p\ge\frac{n}{n-1}.\\
\end{array}
\right.
\end{equation}

As $D\in L^\infty(\R^n)$, there exists $\la_0>0$ such that $|D|\le \la_0$ and then we choose $\phi_{\la_0}$ which is ensured by
Lemma \ref{lem2}. Let
$\Psi(t, x)=\eta_T^{2p'}\Phi (t, x)=\eta_T^{2p'}e^{-\la_0 t}\phi_{\la_0}(x)$ be the test function, 
as for \eqref{middle} and \eqref{large_p}, we get from \eqref{eq-key} and Lemma \ref{keyinq}  that
\begin{equation}\label{T12step3}
\begin{aligned}
&\int_0^T\int_{\R^n}|u|^p\Psi  dxdt + C_2(f,g) \vep\\
\les &
\int_{T/2}^T\int_{B_{t+R}}|u| \eta_T^{2(p'-1)}
\frac {2+|D|} T\Phi dxdt\\
\les & T^{-1}
\left(\int_{T/2}^T\int_{\R^n}\eta_T^{2p'}|u|^p dxdt\right)^{\frac1p}
\left(\int_{T/2}^T\int_{B_{t+R}} \Phi^{p'} dxdt\right)^{\frac1{p'}}
\\
\le&CT^{-1+(n-\frac{n-1}{2}p')\frac{1}{p'}}\left(\int_{T/2}^T\int_{\R^n}\eta_T^{2p'}|u|^p dxdt\right)^{\frac1p},\\
\end{aligned}
\end{equation}
which yields
\begin{equation}\label{T12step4}
\begin{aligned}
\vep^pT^{n-\frac{n-1}{2}p}=\vep^pT^{p-(n-\frac{n-1}{2}p')(p-1)}
\lesssim \int_{T/2}^T\int_{\R^n}\eta_T^{2p'}|u|^p dxdt \ .
\end{aligned}
\end{equation}
Based on \eqref{T12step2} and \eqref{T12step4}, we obtain the first and second lifespan estimates in \eqref{lifespan1b} in Theorem \ref{thmStrauss}. 

\subsection{Critical case}
Turning to the critical case, $p=p_S(n)$,
the proof  is parallel to that in \cite{LT20}, which heavily relies on Lemma \ref{elp}. 
%equipped with the Lemma \ref{elp}, the same argument as that in \cite{LT20} will give the proof. 
For completeness, we present a proof here.

Based on the family of test functions $\phi_{\lambda}$, with $\la\in (0,1]$, satisfying
\begin{equation}\label{elpphila}
\begin{aligned}
\Delta \phi_\lambda-\la D(x)\phi_\la=\lambda^2\phi_\la,~~~x\in \R^n\ ,
\end{aligned}
\end{equation}
we construct a new class of test functions, with parameters $q>0$,
\[
b_q(t, x)=\int_0^1e^{-\lambda t}\phi_{\lambda}(x)\lambda^{q-1}d\lambda\ .
\]
The magic of the test functions $b_q(t, x)$ lie on the facts that they satisfy
\begin{equation}\label{bq1}
\partial^2_tb_q-\Delta b_q-D(x)\partial_tb_q=0,~~~ \partial_tb_q=-b_{q+1}\ ,
\end{equation}
and enjoy the asymptotic behavior for $n\ge 2$ and $r\le t+R$
\begin{equation}\label{bq2a}
b_q(t, x)\gtrsim(t+R)^{-q}, \ q>0\ ,
\end{equation}
and
\begin{equation}\label{bq2}
b_q(t, x)\lesssim\left\{
\begin{array}{ll}
(t+R)^{-q}\ &\mbox{if}\ 0<q<\frac{n-1}2,\\
(t+R)^{-\frac{n-1}2}(t+R+1-|x|)^{\frac{n-1}{2}-q}\ &\mbox{if}\ q>\frac{n-1}2.\\
\end{array}
\right.
\end{equation}
Based on \eqref{bq1}-\eqref{bq2}, the same argument in
\cite{LT20} will yield a proof for the last lifespan of Theorem \ref{thmStrauss}.

\subsubsection{Estimates of the test functions:  \eqref{bq2a} and  \eqref{bq2} }
In \cite{LT20}, the asymptotic behavior \eqref{bq2} was proved by employing the property of the hypergeometric function, when
 $\be>2$ and $n\ge 3$. In the following we will use a relatively simpler way to show it, in the general case
  $\be>1$ and $n\ge 2$,
  inspired by the method in \cite{LLWW}.

We first consider the lower bound \eqref{bq2a}. From the definition of $b_q$ we know
\begin{equation}\label{bq3}
\begin{aligned}
b_q(t, x)&\gtrsim\int_{\frac{1}{2(t+R)}}^{\frac{1}{t+R}}e^{-\lambda t}\phi_{\lambda}(x)\lambda^{q-1}d\lambda\\
&\gtrsim\int_{\frac{1}{2(t+R)}}^{\frac{1}{t+R}}e^{-\lambda (t+R)}\lambda^{q-1}d\lambda\\
&\gtrsim(t+R)^{-q}\int_{\frac12}^{1}e^{-\theta}\theta^{q-1}d\theta\\
&\gtrsim(t+R)^{-q},\\
\end{aligned}
\end{equation}
where we used the fact $\phi_{\lambda}\thicksim 1$ when $r\lambda\le\frac{r}{t+R}\le 1$ by \eqref{510}.

For the upper bound \eqref{bq2}, we divide the proof into two parts: $r\le \frac{t+R}{2}$ and $\frac{t+1}{2}\le r\le t+R$. For the former case, we have
\begin{equation}\label{bq4}
\begin{aligned}
b_q(t, x)&\lesssim\int_{0}^{1}e^{-\frac{\lambda(t+R)}{2}}(1+\lambda r)^{-\frac{n-1}{2}}\lambda^{q-1}d\lambda\\
&\lesssim\int_{0}^{1}e^{-\frac{\lambda(t+R)}{2}}\lambda^{q-1}d\lambda\\
&\lesssim(t+R)^{-q}\int_{0}^{\infty}e^{-\theta}\theta^{q-1}d\theta\\
&\lesssim(t+R)^{-q}.\\
\end{aligned}
\end{equation}
If $\frac{t+R}{2}\le r\le t+R$ and $0<q<\frac{n-1}{2}$, it is clear that
\begin{equation}\label{bq5}
\begin{aligned}
b_q(t, x)&\lesssim\int_{0}^{1}(1+\lambda(t+R))^{-\frac{n-1}{2}}\lambda^{q-1}d\lambda\\
&\lesssim(t+R)^{-q}\int_{0}^{\infty}(1+\theta)^{-\frac{n-1}{2}}\theta^{q-1}d\theta\\
&\lesssim(t+R)^{-q}.
\end{aligned}
\end{equation}
For the remaining case: $\frac{t+R}{2}\le r\le t+R$ and $q>\frac{n-1}{2}$, we see that
\begin{equation}\label{bq6}
\begin{aligned}
b_q(t, x)&\lesssim\int_{0}^{1}e^{-\lambda(t+R+1-r)}(\lambda(t+R))^{-\frac{n-1}{2}}
\lambda^{q-1}d\lambda\\
&\lesssim(t+R)^{-\frac{n-1}{2}}\int_{0}^{1}e^{-\lambda(t+R+1-r)}
\lambda^{q-1-\frac{n-1}{2}}d\lambda\\
&\lesssim(t+R)^{-\frac{n-1}{2}}(t+R+1-r)^{\frac{n-1}{2}-q}\int_0^{\infty}e^{-\theta}
\theta^{q-1-\frac{n-1}{2}}d\theta\\
&\lesssim(t+R)^{-\frac{n-1}{2}}(t+R+1-r)^{\frac{n-1}{2}-q}\ .
\end{aligned}
\end{equation}

\subsubsection{Proof}
With $b_q$ and its asymptotic behavior in hand,
we use $\Psi_T=\eta_T^{2p'}b_q$ as the test function, which gives us
$$
\begin{aligned} & \int_0^T\int_{\R^n}\eta_T^{2p'}b_q|u|^pdxdt\\
=&\int_0^T\int_{\R^n}\left(\partial_t^2u-\Delta u+D(x) \partial_tu\right)b_q\eta_{T}^{2p'}dxdt\\
\lesssim& \int_0^T\int_{\R^n}|u|\left(|2\partial_tb_q\partial_t\eta_T^{2p'}|
+|b_q\partial_t^2\eta_T^{2p'}|
+\Big| D(x) b_q\partial_t\eta_T^{2p'}\Big|\right)dxdt\\
\le & (\int_{\frac T 2}^T\int_{\R^n}|u|^p\Psi_T dxdt)^{\frac 1 p}
\left(\int_{\frac T 2}^T\int_{B_{t+R}}
b_q^{-\frac {p'}p}\left[(T^{-1}
+D )b_q
+ b_{q+1} \right]^{p'}T^{-p'} 
dxdt\right)^{\frac 1 {p'}}\ .
%\\\triangleq& I+II+III,
\end{aligned}
$$
Let
$q=\frac{n-1}{2}-\frac1p$.
As $p>(n+1)/(n-1)$, $p'<(n+1)/2\le n$, $D\le \mu (1+r)^{-\be}$, we have
\begin{eqnarray*}
\int_{\frac T 2}^T\int_{B_{t+R}}
(T^{-1}
+D)^{p'}T^{-p'} b_q dxdt
 & \les & \int_{\frac T 2}^T\int_{B_{t+R}}
 (t+R)^{\frac 1p-\frac{n-1}2}
  T^{-p'}(1+r)^{-p'}dxdt
 \\
   & \les &  T^{\frac 1p-\frac{n-1}2-2p'+n+1}=1
\end{eqnarray*}
where we used \eqref{bq2}  and the fact that
\[
\frac 1p-\frac{n-1}2-2p'+n+1=
\frac 1p+\frac{n-1}2-\frac{2}{p-1}=0
\]
for $p=p_S(n)$. 
By \eqref{bq2a}-\eqref{bq2}, we see that
$$b_{q}\simeq (t+R)^{-q}, b_{q+1}\les 
(t+R)^{-\frac{n-1}2}(t+R+1-|x|)^{\frac{n-1}2-q-1}\simeq
(t+R)^{-\frac{n-1}2}(t+R+1-|x|)^{-\frac{1}{p'}}\ ,$$
and so
\begin{eqnarray*}
 &  &\int_{\frac T  2}^T\int_{B_{t+R}}
 b_{q+1}^{p'}b_q^{-\frac{p'}p}T^{-p'}dxdt\\
 & \les & \int_{\frac T2}^T\int_{B_{t+R}}
 (t+R)^{-\frac{n-1}2-\frac{1}{p(p-1)}}
 (t+R+1-|x|)^{-1}
 (t+R)^{-(\frac{n-1}2-q)\frac{1}{p-1}}
  T^{-p'}dxdt\\
   & \les & \int_{\frac T 2}^T\int_0^{{t+R}}
 T^{\frac{n-1}2-\frac{1}{p(p-1)}-p'}
 (t+R+1-r)^{-1}
drdt\\
  & \les & \int_{\frac T 2}^T\int_0^{{t+R}}
 T^{-1}
 (t+R+1-r)^{-1}
drdt\les \ln T\ .
\end{eqnarray*}
In conclusion, we have
\beeq \label{eq-0217}\int_0^T\int_{\R^n}\eta_T^{2p'}b_q|u|^pdxdt \les (\ln T)^{1/p'} Z(T)^{1/p}
\ ,\eneq
where
$$Z(T)\triangleq\int_{T/2}^T\int_{\R^n}|u|^pb_q \eta_T^{2p'} dxdt\le \int_0^T\int_{\R^n}|u|^pb_q \eta_T^{2p'} dxdt\triangleq X(T)\ .$$

To relate $Z$ and $X$, and recall the critical nature of the situation, let
$Z=TY'$ and $Y(2)=0$, then
%For $M\in [2, T]\subset [2, T(\ep))$, we set
\[
Y(M)=\int_2^M Z(T) T^{-1}dT ,
\]
and
\begin{equation*}
\begin{aligned}
Y(M)
=&\int_2^M
\left(\int_{T/2}^T\int_{\R^n}b_q|u|^p(\eta_T(t))^{2p'}dx dt\right) T^{-1}dT\\
\le &\int_{1}^M\int_{\R^n}b_q|u|^p\int_{t}^{\min(M,2t)}(\eta_T(t))^{2p'}T^{-1}dT dx dt\\
=&\int_{1}^M\int_{\R^n}b_q|u|^p\int^{1}_{\max(t/M, 1/2)}(\eta(s))^{2p'}s^{-1}ds dx dt\\
\le &\int_{1}^M\int_{\R^n}b_q|u|^p (\eta(t/M))^{2p'}\int^{1}_{1/2}s^{-1}ds dx dt\\
\le &\ln 2 \int_{1}^M\int_{\R^n}b_q|u|^p (\eta(t/M))^{2p'}  dx dt\le
\ln 2\int_0^M\int_{\R^n}|u|^pb_q \eta_T dxdt\les  X(M)\ ,
\end{aligned}
\end{equation*}
where we used the assumption that $\eta$ is decreasing.
Thus, recalling \eqref{eq-0217}, we have
\beeq \label{Yp} Y(T) \les (\ln T)^{1/p'} (TY'(T))^{1/p},\ 
TY'(T)\ge c Y ^p (\ln T)^{1-p}, \forall T\in [2, T_\vep )\ .
\eneq

In addition, by \eqref{T12step4} and \eqref{bq2a}, we have
\begin{equation}%\label{thmcristep1}
\label{Ycrideriv}
TY'=Z(T)\triangleq\int_{T/2}^T\int_{\R^n}|u|^pb_q \eta_T ^{2p'} dxdt 
\gtrsim \e^p T^{n-\frac{n-1}{2}p-q}=\e^p.
\end{equation}
where we used the fact
\[
n-\frac{n-1}{2}p=q=\frac{n-1}{2}-\frac1p
\]
for $p=p_S(n)$.

Equipped with  \eqref{Ycrideriv} and \eqref{Yp}, we could apply
Lemma 3.10 in \cite{ISWa} to conclude the last lifespan in \eqref{lifespan1b}.
Actually,  by \eqref{Ycrideriv},
integration from $2$ to $T>4$ yields 
\beeq\label{0217-1}Y(T)\ge Y(2)+c\e^p (\ln T-\ln 2)\gtrsim \e^p \ln T,\ \forall T\in (4, T_\e)\ .\eneq
Similarly,  for \eqref{Yp}, integration from $T_1$ to $T_2>T_1$ gives us
%\begin{eqnarray*}Y(T_2)^{1-p}&\le  & Y(T_1)^{1-p}-c(p-1) \int_{\ln T_1}^{\ln T_2} \tau^{1-p}d\tau \\& \le &  Y(T_1)^{1-p}-c(p-1)  (\ln T_2)^{1-p}\ln (T_2/T_1)\end{eqnarray*}
$$Y(T_2)^{1-p}\le  Y(T_1)^{1-p}-c(p-1) \int_{\ln T_1}^{\ln T_2} \tau^{1-p}d\tau 
\ , \forall 2<T_1< T_2<T_\e \ .$$
As $Y(T)\ge 0$, letting $T_2\to T_\e$, and using \eqref{0217-1}, we see that
$$ \int_{\ln T_1}^{\ln T_\vep} \tau^{1-p}d\tau \les Y(T_1)^{1-p}\les 
\e^{-p(p-1)} (\ln T_1)^{1-p}, \forall 4<T_1< T_\e\ .$$
Setting $T_1=\sqrt T_\e$, it follows that
$$  \ln T_\vep \les 
\e^{-p(p-1)}\ ,$$
which gives us the desired lifespan for the critical case, in \eqref{lifespan1b}.

\begin{center}
ACKNOWLEDGMENTS
\end{center}
The first author is supported by NSF of Zhejiang Province(LY18A010008) and NSFC 11771194.
The second and fourth author were supported in part by NSFC 11971428.

%\clearpage
%\end{CJK*}

\end{document}